\documentclass{amsart}
\usepackage{style}
\renewcommand{\L}{\mathrm{L}}
\renewcommand{\H}{\mathrm{H}}

\title{Picard groups of quotient ring spectra}
\author{Ishan Levy}
\address{Department of Mathematics, Institute of Advanced Studies, USA}
\hemail{ishanl@ias.edu}
\author{Guchuan Li}
\address{School  of Mathematical Sciences, Peking University, Beijing, China}
\hemail{liguchuan@math.pku.edu.cn}
\author{Ningchuan Zhang}
\address{Department of Mathematics, Indiana University Bloomington, USA}
\hemail{nz7@iu.edu}
\date{}

\begin{document}
	\begin{abstract}
		We develop tools to study Picard groups of quotients of ring spectra by a finitely generated ideal, which we use to show that $\Pic(\E_n/I) = \Z/2$, where $\E_n$ is a Lubin--Tate theory and $I$ is an ideal generated by suitable powers of a regular sequence. We apply this to obtain spectral sequences computing Picard groups of $\mathrm{K}(n)$-local generalized Moore algebras, and make some preliminary computations including the height $1$ case.
	\end{abstract}
	\maketitle	

	\section{Introduction}
	When studying the Picard group of a commutative ring $R$, i.e. the group of invertible modules, a common technique is to break up the problem, by choosing an element $v$, and studying the problem after inverting $v$ and after $v$-completion, which is a form of Beauville--Laszlo gluing. Understanding line bundles after $v$-completion is aided by the fact that the group of invertible objects is nil-invariant, so that the problem can often be reduced to a question about the smaller quotient ring $R/v$.
	
	In higher algebra, these techniques do not carry over in as straightforward a manner. Given a commutative ring spectrum $R$, we may try to understanding its space of invertible modules $\cPic(R)$ in the same way: given an element $v \in \pi_*R$, we may form the localization and completions of the category of $R$-modules with respect to $v$, which formally satisfy a form of Beauville--Laszlo gluing. However, it is less clear that these can be related to questions about the quotient $R/v$.
	
	The first step towards this was in \cite{LZ_profin_Picard}, whose techniques show that $$\cPic(\Mod(R)^{\wedge}_v) \simto \lim_n\cPic(\Mod(R/v^n)),$$ where $R/v^n$ has an $\EE_2$-ring structure for $n$ suitably large as in \cite[Theorem 1.5]{Burkland_mult_Moore}. The main goal of this paper is to provide tools to compute Picard groups of rings such as $R/v^n$ as $n$ varies.
	
	A sample result is the following, which we use to give an upper bound on $\Pic(R/v^n)$ in terms of smaller values of $n$:
	
	\begin{thm}[{\ref{cor:pic_inj}}]
		Let $R$ be an $\Ebb_3$-ring. Suppose $v\in \pi_*(R)$ is an element such that $R/v$ admits a left unital multiplication in $R$-modules. If $n\geq 3$ and $\pi_{-1-n|v|}(R/v)=0$, then the base change map $\Pic(R/v^{n+1})\to \Pic(R/v^n)$ is injective where $R/v^{n+1}\to R/v^n$ is made into an $\Ebb_2$-$R$-algebra map using \cite[Theorem 1.5]{Burkland_mult_Moore}.
	\end{thm}
	
	Because $R/v^n$ is often not an $\EE_2$-ring, we work more generally with the notion of the \textit{relative module space} (\Cref{definition:relativemodulespace}) of a map $f\colon  R \to S$ of $\EE_1$-rings, which generalizes the fiber of the map of Picard spaces in the case that $R,S$ are $\EE_2$-rings and $f$ is conservative on module categories. We prove versions of the above theorem in this more general setting in \Cref{prop:pullbackupperbound} and \Cref{thm:mod_quot}.
	
	For $\EE_\infty$-rings $R$ with complete regular local even periodic homotopy rings, $\Pic(R)$ was understood in \cite{Baker-Richter_invertible_modules,MS_Picard}. Our results allow us to compute the Picard group of such kinds of rings:
	
	\begin{thm}[{\ref{iteratedquotient}}]
		Let $R$ be an $\EE_2$-ring such that $\pi_*R$ is even and a complete regular local graded ring with maximal ideal $\mfrak$ generated by the regular sequence $x_1,x_2,\cdots,x_m$. For any $\EE_2$-refinement of the $\EE_1$-$R$-algebra structure on $R/(x_1^{a_1},\cdots,x_m^{a_m})$ from \Cref{construction:quotientregularsequence}, $\Pic(R/(x_1^{a_1},\cdots,x_m^{a_m}))$ is generated by $\Sigma R$.
	\end{thm}
	
	We may apply the above to $R = \E_n$, the Lubin--Tate theory, where $\pi_*\E_n$ is local with the maximal ideal $(p,v_1,\cdots,v_{n-1})$. If $M$ is a generalized Moore $\EE_2$-algebra\footnote{By a generalized Moore ($\Ebb_2$-)algebra, we mean a generalized Moore spectrum $M$ that admits an ($\Ebb_2$-)algebra structure.} as in \cite{Burkland_mult_Moore}, the work of {\cite{LZ_profin_Picard,mor2023picard}} provides a profinite descent spectral sequence computing its Picard group with signature:
	\begin{equation*}
		^{\Pic} E_2^{s,t}=\H_c^s(G;\pi_t\cPic_{\K(n)}(\E_n\otimes M))\Longrightarrow \pi_{t-s}\cPic_{\K(n)}(\E_n^{hG}\otimes M), \qquad t-s\ge 0.
	\end{equation*}
	Our results then provide an understanding of the $E_2$-page of this spectral sequence, which in principle give computational access to $\Pic(\L_{\K(n)}M)$.
	
	In \Cref{section:computations}, we give examples of computations using the above techniques and results. In particular, we compute  the Picard groups of the following quotients of (Tate) $\K$-theories:
	\begin{enumerate} 
		\item $\Pic(\KU\llb q\rrb/q^k)=\Z/2$ and $\Pic(\KO\llb q\rrb/q^k)=\Z/8$ for $k\geq 1$ (\Cref{prop:Pic_TateK}).
		\item $\Pic(\KO/\eta^k) = \ZZ/4$ for $k=2,3$ (\Cref{thm:PicKOeta2}) and is $ \ZZ/8$ for $k=4,5$ (\Cref{thm:eta4}).
		\item
		$\Pic(\KO/2^k)=\Z/8$ when $k\ge 6$ 
		(\Cref{thm:PicKO2k}). 
	\end{enumerate}
	At a general height, we obtain the following qualitative result of Picard groups of $\K(n)$-local generalized Moore algebras: 
	\begin{thm}[\ref{thm:Moore_Pic}]
		Let $M=\mathbb{S}/(p^{d_0},v_1^{d_1},\cdots, v_{n-1}^{d_{n-1}})$ be a generalized Moore algebra of type $n$. 
		\begin{enumerate}
			\item The Picard group $\Pic(\L_{\K(n)}M)$ is finite.
			\item When $2p-1> n^2$ and $(p-1)\nmid n$, the Picard group $\Pic(\L_{\K(n)}M)$ is algebraic in the sense of a short exact sequence:
			\begin{equation*}
				\begin{tikzcd}
					0\to \H^1_c(\Gbb_n;\pi_0(\E_n\otimes  M)^\times)\rar & \Pic(\L_{\K(n)}M)\rar &\Z/2 \to 0.
				\end{tikzcd}
			\end{equation*}
		\end{enumerate}
	\end{thm}
	At height $1$, we compute the Picard groups of $\K(1)$-local Moore algebras:
	\begin{thm}[\ref{thm:Pic_S0pk}, \ref{thm:Pic_S02k}]
		\begin{equation*}
			\Pic\left(\Sbb_{\K(1)}/p^k\right)
			=\begin{cases}\Z/2^{k-1}\oplus\Z/4\oplus \Z/2, & p=2\text{ and }k\ge 6;\\
				\Z/(2(p-1)p^{k-1}), & p>2\text{ and }k\ge 3.
			\end{cases}
		\end{equation*}
	\end{thm}
	Note that when  $p=2, 2\le k\le 4$ or $p>2, k=2$, it is unknown whether $\Sbb_{\K(1)}/p^k$ has an $\Ebb_2$-ring structure.  For $p=2$ and $k=5$, we can only determine $\Pic\left(\KO/32\right)$ and $\Pic\left(\Sbb_{\K(1)}/32\right)$ up to a possible extension by $\Z/2$ (\Cref{rem:PicKO32}).
	\subsection*{Future directions}
	For all primes and heights, the algebraic $\K(n)$-local Picard group has been completely computed by Barthel--Schlank--Stapleton--Weinstein in \cite{BSSW_Picard}. In particular, their result implies that when $2p-1>n^2$ and $(p-1)\nmid n$, we have $\Pic\left(\Sp_{\K(n)}\right)=\Zp\oplus \Zp\oplus\Z/(2p^n-2)$, generated by the regular and the determinant $\K(n)$-local spheres. Analyzing the orders of their images in $\Pic(\L_{\K(n)}M)$ for a generalized Moore algebra $M$, one can ask:
	\begin{quest}[\ref{quest:Pic_M}]
		Let $M=\Sbb/(p^{d_0},\cdots, v_{n-1}^{d_{n-1}})$ be a generalized Moore $\Ebb_2$-algebra of type $n$ such that $v_n^{p^N}$ is the minimal periodicity of $M$. Under the assumptions $2p-1>n^2$ and $(p-1)\nmid n$, is the map $\Pic\left(\Sp_{\K(n)}\right) \longrightarrow \Pic(\L_{\K(n)}M)$ surjective so that $ \Pic(\L_{\K(n)}M)\cong \Z/p^N\oplus\Z/p^{d_0-1}\oplus\Z/(2p^n-2)$? 
	\end{quest}
	More generally, one might wonder about the surjectivity of the base change map $\Pic(R/v^{n+1})\to \Pic(R/v^n)$, whereas only injectivity is addressed in \Cref{thm:mod_quot}. In all examples studied in this paper, the map is always surjective, but we would be surprised if this were always true. 
	\begin{quest}
		Is $\Pic(R/v^{n+1})\to \Pic(R/v^n)$ always surjective? 
	\end{quest}
	A related topic is the (relative) Brauer groups of quotient $\Ebb_3$-ring spectra.  An interesting question to further study is whether there are versions of the tools developed here to study relative Brauer groups, for example for quotient maps of the form $R/v^{n+1}\to R/v^n$.
	
	\subsection*{Notation and Conventions}
	We use category to mean $\infty$-category in the sense of Joyal and Lurie.
	
	\begin{enumerate}
		\item Given a category $\cC$, we use $\cC^{\simeq}$ to denote the core groupoid of objects and isomorphisms in $\cC$.
		\item Given a monoidal category $\cC$, we use $\cPic(\cC)$ to denote the space (groupoid) of invertible objects, and $\Pic(\cC)$ to denote $\pi_0(\cPic(\cC))$, which is the Picard group of $\cC$.
		\item We use $\Map_\cC(a,b)$ to denote the mapping space, and use $\map_{\cC}(a,b)$ to denote the mapping spectrum in a stable category. 
	\end{enumerate}
	\subsection*{Acknowledgments}
	We want to thank Ko Aoki, Tobias Barthel, Agn\`es Beaudry, Robert Burklund, David Gepner, Paul Goerss, Hans-Werner Henn, Michael Mandell, Akhil Mathew, Lennart Meier, Itamar Mor, and  Vesna Stojanoska for many helpful and supportive discussions and comments. We would also like to thank the anonymous referee for the helpful comments in improving the paper. 
	
	I. Levy was supported by the Clay Research Fellowship.  Some of the work and revision were done when G. Li and N. Zhang were at the Max Planck Institute for Mathematics. They would like to thank the MPIM for its hospitality and support. N. Zhang was partially supported by the NSF Grant DMS-2348963 (formerly DMS-2304719). 
	\section{Tools for studying Picard groups}\label{section:tools}
	In this section, we develop tools to study Picard groups. In doing so, it will be useful to work with categories that are not monoidal, such as perfect module categories of $\EE_1$-rings. Although $\EE_1$-rings do not have a reasonable notion of a Picard group, there is a relative version of the notion that is often well behaved, namely \Cref{definition:relativemodulespace} below. It will be convenient to formulate some of our definitions even more generally, namely in terms of pointed small idempotent complete stable $\infty$-categories. We use $\Cat^{\perf}$ to denote the category of small idempotent complete stable $\infty$-categories, and let $\Cat_*^{\perf}$ denote the category of $\EE_0$-algebras in $\Cat^{\perf}$, that is, the category of $\cC\in \Cat^{\perf}$ equipped with a specified object $c \in C$. As an example, for an $\EE_1$-ring $R$ the category $\Perf(R)$ of perfect left $R$-modules naturally lives in $\Cat_*^{\perf}$, as it is equipped with the unit object $R$.
	
	\subsection{Relative module spaces and pullbacks}
	\begin{defn}\label{definition:relativemodulespace}
		Let $f\colon  \cC \to \cD$ be a map in $\Cat_*^{\perf}$. The \textit{relative object space} of $f$ is $\cC^{\simeq} \times_{\cD^{\simeq}}*$, where we use the specified object of $\cD$. The specified object of $\cC$ makes this into a pointed space. In the case $f$ comes from base change along a map $g\colon R \to S$ of $\EE_1$-rings, we call this the \textit{relative module space} of $g$, and denote it $\Perf(R\mid S)$.
	\end{defn}
	In other words, given $g\colon R \to S$, a point in the relative module space of $R$ over $S$ is a perfect $R$-module $M$ and an equivalence $M\otimes_{R}S \cong S$.
	
	The relative module space is closely related to Picard groups when the functor is conservative:
	
	\begin{lem}\label{lemma:relativepicardgroups}
		Let $f\colon  \cC \to \cD$ be a map of rigid\footnote{This means that all objects are left and right dualizable.} $\EE_{k}$-monoidally\footnote{This means that the tensor product commutes with finite colimits.} stable idempotent complete categories for $k\geq1$. Then the relative object space of $f$ is naturally an $\EE_{k}$-space. If $f$ is furthermore conservative, then there is a fiber sequence of $\EE_{k}$-spaces
		$$\cC^{\simeq} \times_{\cD^{\simeq}}* \longrightarrow  \cPic(\cC) \longrightarrow  \cPic(\cD).$$
	\end{lem}
	
	\begin{proof}
		We note that when $f$ is $\EE_k$-monoidal, the spaces involved in defining the relative module space are all $\EE_{k}$-spaces, as are the maps between them. Thus $\cC^{\simeq} \times_{\cD^{\simeq}}*$ is an $\EE_{k}$-space.
		
		The space $\cC^{\simeq} \times_{\cD^{\simeq}}*$  is defined by the fiber sequence $$\cC^{\simeq} \times_{\cD^{\simeq}}* \longrightarrow \cC^\simeq \longrightarrow  \cD^\simeq.$$
		Since $\cPic(\cC)$ is a collection of components inside $\cC^{\simeq}$, it suffices to see that the image of $\cC^{\simeq} \times_{\cD^{\simeq}}* \to \cC^{\simeq}$ is contained in $\cPic(\cC)$. In other words, we need to show that if $M \in \cC$ such that $f(\cC)$ is invertible, then $M$ is invertible. But $M$ is dualizable, so invertibility is the condition that the maps $M\otimes M^{\vee} \to R$ and $R \to M^{\vee}\otimes M$ are equivalences. By conservativity of $f$, this can be checked after base change to $S$, where it follows because $M\otimes_RS\cong S$.
	\end{proof}
	
	\begin{rem}
		If $R \to S$ is a map of $\EE_1$-rings, even though the Picard group does not make sense, there is still a fiber sequence	
		$$\Perf(R\mid S)' \longrightarrow \mathrm{BGL}_1(R) \longrightarrow \mathrm{BGL}_1(S),$$
		where $\Perf(R\mid S)'$ is the subspace of $\Perf(R\mid S)$ where the underlying $R$-module is isomorphic to $R$.
	\end{rem}
	
	The first result we prove is about relative module spaces of pullbacks of $\EE_1$-rings. In fact we prove a more general result allowing us to understand relative module spaces of a pullback of two $\EE_1$-rings along a \textit{bimodule}, which is described in \Cref{const:pullbackbimodule}.
	\begin{defn}[{\cite[Definition 5]{Tamme_excision2018}}]\label{defn:ori_pb}
		Let $\Acal\xrightarrow{p} \Ccal\xleftarrow{q} \Bcal$ be a diagram of $\infty$-categories. The \textit{oriented pullback} of this diagram $\Acal\vec{\times}_{p,\Ccal,q}\Bcal$ (denoted simply by $\Acal\vec{\times}_{\Ccal}\Bcal$ when clear) is the pullback of $\infty$-categories:
		\begin{equation*}
			\begin{tikzcd}
				\Acal\vec{\times}_{p,\Ccal,q}\Bcal\rar\dar\ar[dr,phantom,"\lrcorner",very near start]& \Fun(\Delta^1,\Ccal)\dar\\
				\Acal\times \Bcal\rar["p\times q"]&\Ccal\times \Ccal. 
			\end{tikzcd}
		\end{equation*}
		Objects in $\Acal\vec{\times}_{p,\Ccal,q}\Bcal$ are triples $(X,Y,f\colon p(X)\to q(Y))$ where $X\in \Acal$ and $Y\in \Bcal$.
	\end{defn}
	
	\begin{rem}
		The pullback $\Acal\times_{\Ccal}\Bcal$ is the full subcategory of the oriented pullback $\Acal\vec{\times}_{\Ccal}\Bcal$ spanned by objects where the comparison map $f$ is an equivalence.
	\end{rem}
	We recall the notion of a Milnor context (over the sphere) from \cite{LT_pushouts}:
	\begin{defn}\label{definition:milnorcontext}
		A Milnor context is a triple $(R,S,M)$ where $R,S$ are $\EE_1$-rings, and $M$ is a unital $R$-$S$-bimodule. A map of Milnor contexts $(R,S,M) \to (R',S',M')$ is a pair of maps of $\EE_1$-rings $R \to R', S \to S'$ and an $R$-$S$-bimodule map $M \to M'$ where $M'$ is viewed as an $R$-$S$-bimodule via the forgetful functor.
	\end{defn}
	\begin{const}[{\cite[Construction 2.5]{LT_pushouts}}]\label{const:pullbackbimodule}
		Given a Milnor context $(R,S,M)$, we can form the oriented pullback \begin{equation*}
			\Mod(R)\vec{\times}_M\Mod(S):=\Mod(R)\vec{\times}_{\id,\Mod(R),M\otimes_S(-)}\Mod(S)
		\end{equation*}using the functor $M\otimes_S(-)\colon \Mod(S) \to \Mod(R)$ and the identity functor on $\Mod(R)$. By \cite[Lemma 2.3]{LT_pushouts}, this is a compactly generated category with compact objects given by \begin{equation*}
			\Perf(R)\vec{\times}_M\Perf(S):=\Perf(R)\vec{\times}_{\id,\Mod(R),M\otimes_S(-)}\Perf(S).
		\end{equation*} The unit on $M$ gives rise to an object $I_M=(R,S,R \to M = M\otimes_{S}S)$, making $\Perf(R)\vec{\times}_{M}\Perf(S)$ into an object of $\Cat_*^{\perf}$. 
		
		The pullback ring $R\times_M S$ is defined as the opposite of the endomorphism ring of $I_M$. The Schwede--Shipley theorem \cite[Theorem 7.1.2.1]{Lurie_HA} gives a fully faithful embedding  $\Perf(R\times_MS) \to \Perf(R)\vec{\times}_M\Perf(S)$ in $\Cat_*^{\perf}$. In the case that the unital bimodule $M$ comes from a span of $\EE_1$-rings $R \to M \leftarrow S$, then $R\times_M S$ agrees with the pullback of rings \cite[Remark 2.6]{LT_pushouts}.
	\end{const}
	
	A key ingredient to getting an upper bound on relative module spaces of pullbacks is the following lemma, which describes relative object spaces of certain oriented pullbacks.
	
	\begin{lem}\label{lemma:pullbackupperboundcat}
		Let $(R,S,M) \to (R',S',M')$ be a map of Milnor contexts. Given $N \in \Perf(R), R'\cong R'\otimes_RN$, and $N' \in \Perf(S), S' \cong S'\otimes_SN'$, so that $N,N'$ give rise to points in $\Perf(R\mid R')$ and $\Perf(S\mid S')$, respectively. Let $F$ denote the fiber of the natural map		
		$$(\Perf(R)\vec{\times}_{M}\Perf(S))^{\simeq} \times_{(\Perf(R')\vec{\times}_{M'}\Perf(S'))^{\simeq} }* \longrightarrow  \Perf(R\mid R')\times \Perf(S\mid S')$$ over the point given by $(N,N')$.
		
		Then $F$ is the fiber of the map
		\[\Map_{\Mod(R)}(N,M\otimes_SN')\longrightarrow \Omega^{\infty}M'\] over the unit, where the map sends a map $g\colon N \to M\otimes_SN'$ to the point corresponding to the $R'$-module map $$R' \cong R'\otimes_RN \longrightarrow R'\otimes_RM\otimes_SN' \longrightarrow  M'\otimes_{S}N' \cong M'\otimes_{S'}S'\otimes_{S}N' \cong M'.$$
	\end{lem}
	
	\begin{proof}
		Consider the following cube, where the top and bottom faces are pullback squares by \Cref{defn:ori_pb} of the oriented pullback:
		\setlength{\perspective}{2pt}
		\[\begin{tikzcd}[row sep={40,between origins}, column sep={80,between origins}]
			&[-\perspective] (\Mod(R)\vec{\times}_M \Mod(S))^{\simeq} \ar{rr}\ar{dd}\ar{dl} &[\perspective] &[-\perspective] (\Mod(R)^{\Delta^1})^{\simeq}\vphantom{\times_{S_1}} \ar{dd}\ar{dl} \\[-\perspective]
			(\Mod(R)\times \Mod(S))^{\simeq}  \ar[crossing over]{rr} \ar{dd} & & (\Mod(R)\times \Mod(R))^{\simeq} \\[\perspective]
			& (\Mod(R')\vec{\times}_{M'} \Mod(S'))^{\simeq}  \ar{rr} \ar{dl} & &  (\Mod(R')^{\Delta^1})^{\simeq}\vphantom{\times_{S_1}} \ar{dl} \\[-\perspective]
			(\Mod(R')\times \Mod(S'))^{\simeq}\ar{rr} && (\Mod(R')\times \Mod(R'))^{\simeq} \ar[from=uu,crossing over]
		\end{tikzcd}\]
		Taking vertical fibers and pulling back around our chosen point in $\Perf(R\mid R')\times \Perf(S\mid S')$, we get a pullback square		
		\[\begin{tikzcd}
			F\ar[r]\ar[d] \ar[dr,phantom,very near start,"\lrcorner",end anchor= west]& [30 pt](\Mod(R)^{\Delta^1})^{\simeq}\times_{(\Mod(R')^{\Delta^1})^{\simeq}}*\ar[d]\\
			* \ar[r] &(\Mod(R)\times \Mod(R))^{\simeq}\times_{(\Mod(R')\times \Mod(R'))^{\simeq}}*.
		\end{tikzcd}\]
		Since pullbacks commute with fibers, we can refactor this as the pullback
		\[\begin{tikzcd}
			F\ar[r]\ar[d] \ar[dr,phantom,very near start,"\lrcorner",end anchor= west]& *\ar[d]\\
			(\Mod(R)^{\Delta^1})^{\simeq}\times_{(\Mod(R)\times \Mod(R))^{\simeq}}* \ar[r] &(\Mod(R')^{\Delta^1})^{\simeq}\times_{(\Mod(R')\times \Mod(R'))^{\simeq}}* .
		\end{tikzcd}\]
		However each of the terms in this pullback is a mapping space in $\Mod(R)$ and $\Mod(R')$ respectively, and so one sees that this is the description given in the lemma of $F$ as a fiber of mapping spaces.
	\end{proof}
	
	As a consequence, we get an "upper bound" on relative module spaces of pullbacks:
	
	\begin{prop}\label{prop:pullbackupperbound}
		Let $(R,S,M) \to (R',S',M')$ be a map of Milnor contexts. Given $N \in \Mod(R\mid R'), N' \in \Perf(S\mid S')$, the fiber of the map
		$$\Perf(R\times_M S \mid  R'\times_{M'}S') \longrightarrow \Perf(R\mid R')\times \Perf(S\mid S')$$ over the point given by $N,N'$ is a collection of components inside the fiber of the map
		\[\Map_{\Mod(R)}(N,M\otimes_SN')\longrightarrow \Omega^{\infty}M'\] over the unit, where the map sends a map $g\colon N \to M\otimes_SN'$ to the point corresponding to the $R'$-module map $$R' \cong R'\otimes_RN \longrightarrow R'\otimes_RM\otimes_SN' \longrightarrow M'\otimes_{S'}S'\otimes_{S}N' \cong M'.$$
	\end{prop}
	
	\begin{proof}
		By applying \Cref{lemma:pullbackupperboundcat}, it is enough to see that $\Perf(R\times_MS\mid  R'\times_{M'}S')$ is a collection of components inside $(\Mod(R)\vec{\times}_{M}\Mod(S))^{\simeq} \times_{(\Mod(R')\vec{\times}_{M'}\Mod(S'))^{\simeq} }*$. But this follows because we have the commutative diagram in $\Cat_*^{\perf}$ below, where the horizontal maps are fully faithful:
		\[\begin{tikzcd}[baseline=(LR.base)]
			\Perf(R\times_MS)\ar[r]\ar[d] & \ar[d]\Perf(R)\vec{\times}_M\Perf(S)\\
			\Perf(R'\times_{M'}S')\ar[r]& |[alias=LR]| \Perf(R')\vec{\times}_{M'}\Perf(S')
		\end{tikzcd}.\qedhere\]
	\end{proof}
	
	\subsection{The descent tower}
	
	Our next goal is to apply the above to study quotients of $\EE_1$-rings by powers of ideals. Ultimately, we want to apply this to study Picard groups of towers of the form $R/v^n$ for $v \in \pi_*R$. 
	\begin{defn}\label{definition:E1descenttower}
		Let $f\colon  R \to S$ be a map of $\EE_1$-rings. We define $\EE_1$-rings $D_i(f)$ that fit into a tower we call the \textit{descent tower of $f$}
		$$R \longrightarrow \cdots \longrightarrow D_i(f) \longrightarrow D_{i-1}(f) \longrightarrow \cdots \longrightarrow D_0(f) = S,$$
		where $D_{i+1}(f)$ for $i\ge 0$ is defined as $D_{i}(f)\times_{D_i(f)\otimes_{R}S}S$.
	\end{defn}
	
	\begin{lem}\label{lemma:bimoduleofdescenttower}
		Let $f\colon  R \to S$ be a map of $\EE_1$-rings.
		The underlying tower of unital $R$-bimodules of $D_i(f)$ is the cofiber of the map $I^{\otimes_{R}(i+1)} \xrightarrow{\fib(f)^{\otimes_R i+1}} R$, where $I$ is the fiber of the map $R \to S$.
	\end{lem}
	
	\begin{proof}
		This statement can be proven by induction, and is clear in the case $i=0$. For the inductive step, the fiber of the map $R \to D_{i+1}(f)$ is the total fiber of the square
		\[\begin{tikzcd}
			R\ar[r]\ar[d] & D_i(f)\ar[d]\\
			S\ar[r]&D_i(f)\otimes_RS.
		\end{tikzcd}\]
		
		Taking horizontal fibers and using the inductive hypothesis, we find it is the fiber of the map $I^{\otimes_R i} \to I^{\otimes_R i}\otimes_RS$, which is $I^{\otimes_R (i+1)}$ as claimed.
	\end{proof}
	
	\begin{exmp}\label{example:quotientbyanelement}
		Let $R$ be an $\EE_1$-ring, suppose $v\colon \Sigma^iR \to R$ is an $R$-bimodule map such that the map of $R$-bimodules $R \to R/v$ refines to an $\EE_1$-algebra map. Then by \Cref{lemma:bimoduleofdescenttower}, $D_i(f)$ is an $\EE_1$-ring whose underlying $R$-bimodule is $R/v^{i+1}$.
	\end{exmp}
	
	The following lemma compares our construction to the quotient by powers construction of \cite{Burkland_mult_Moore}:
	
	\begin{lem}\label{lemma:vcompatibledescent}
		If $R$ is an $\EE_3$-ring and $S$ is an $\EE_1$-$R$-algebra, then $D_i(f)$ is an $\EE_1$-$R$-algebra, and is equivalent to the $\EE_1$-$R$-algebra constructed in \cite[Theorem 1.5]{Burkland_mult_Moore} on $R/I^{\otimes_R i+1}$ where $I$ is the fiber of $R \to S$.
	\end{lem}
	
	\begin{proof}
		By the uniqueness statement of \cite[Theorem 1.5]{Burkland_mult_Moore}, it suffices to prove that $D_i(f)$ is $I$-compatible in the sense of \cite[Definition 5.1]{Burkland_mult_Moore}. This can be proven by first observing that the definition of $D_i(f)$ \Cref{definition:E1descenttower} works for any map $f\colon  R \to S$ of $\EE_1$-algebras in a presentably monoidal stable $\infty$-category $\cC$, and is natural in monoidal functors $\cC \to \cD$. Letting $Q$ be as in \cite[Definition 4.4]{Burkland_mult_Moore} for $\cC = \Mod(R)$ and $\one/v$ being $S$, we can now consider the monoidal functor $(-)^{\tau=1}\colon\Def(\Mod(R); Q) \to \Mod(R)$ as in  \cite[Construction 4.6]{Burkland_mult_Moore}.
		
		Using the monoidal functor $\nu\colon \Mod(R) \to \Def(\Mod(R); Q)$ (see loc. cit.), $D_i(\nu(f))$ is a witness of the fact that $D_i(f)$ is $I$-compatible, so we are done.
	\end{proof}
	
	\begin{rem}
		In fact, a version of \Cref{lemma:vcompatibledescent} holds when $R$ is just an $\EE_2$-algebra or when $R \to S$ is just an $\EE_1$-algebra map. However, to state it, one needs to extend the result of \cite[Theorem 1.5]{Burkland_mult_Moore} to the setting of monoidal categories. Such an extension will appear in the final version of \cite{Burkland_mult_Moore}.
	\end{rem}

	Finally, we note that the maps in the descent tower are conservative on module categories.\footnote{In fact, they are descendable in a suitable sense, but we do not pursue this here.}
	
	\begin{lem}\label{lemma:descenttowerdescendable}
		Let $f\colon  R \to S$. Then the maps $f_{i,j}\colon D_i(f) \to D_j(f)$ for $i>j$, induce conservative functors on module categories.
	\end{lem}
	
	\begin{proof}
		Because conservative functors are closed under composition, it suffices to show the lemma for the map $f_{i+1,i}$. If $M$ is a $D_{i+1}(f)$-module such that $D_{i}(f)\otimes_{D_{i+1}(f)}M$ vanishes, then both $S\otimes_{D_{i+1}(f)}M$ and $D_i(f)\otimes_{R}S\otimes_{D_{i+1}(f)}M$ vanish, so by the pullback square defining $D_{i+1}(f)$, we see that $M$ vanishes.
	\end{proof}    
	
	\begin{thm}\label{thm:mod_quot}
		Let $f\colon  R\to S$ be a map of $\EE_1$-rings. Then $\Perf(D_{i+1}(f) \mid D_{i}(f))$ is a union of components inside $\Omega^{\infty}(\Sigma I^{\otimes_R i+1}\otimes_RS)$. In particular, if $\pi_{-1} I^{\otimes_R i+1}\otimes_RS=0$, then $\Perf(D_{i+1}(f) \mid D_{i}(f))$ is connected.
	\end{thm}
	
	\begin{proof}
		Applying \Cref{prop:pullbackupperbound} in the case of the map of Milnor contexts $(D_i(f),S,D_i(f)\otimes_RS)\to (D_i(f),S,S)$, we find that it is enough to show that the fiber of the map \[\Omega^{\infty}(D_i(f)\otimes_RS)\longrightarrow \Omega^{\infty}S\] over the unit is connected.
		
		This map is a map of spectra with a section coming from the unit map, so its fiber over all points are equivalent and equivalent to $\fib(D_i(f)\otimes_RS \to S)$. But this fiber is $\Sigma\fib(S \to D_i(f)\otimes_RS)$, which is equivalent to $\Sigma I^{\otimes_R i+1}\otimes_RS$ by \Cref{lemma:bimoduleofdescenttower}.
	\end{proof}
	
	To demonstrate the purpose of \Cref{thm:mod_quot}, we have the following corollary.
	\begin{cor}\label{cor:pic_inj}
		Let $R$ be an $\Ebb_3$-ring. Suppose $v\in \pi_*(R)$ is an element such that $R/v$ admits a left unital multiplication in $R$-modules. If $n\geq 3$ and $\pi_{-1-n|v|}(R/v)=0$, then the base change map $\Pic(R/v^{n+1})\to \Pic(R/v^n)$ is injective where $R/v^{n+1}\to R/v^n$ is made into an $\Ebb_2$-$R$-algebra map using \cite[Theorem 1.5]{Burkland_mult_Moore}.
	\end{cor}
	\begin{proof}
		By \Cref{lemma:relativepicardgroups}, it is enough to show that the relative module space $\Perf(R/v^{n+1} \mid R/v^n)$ is connected. 
		By \Cref{thm:mod_quot}, the kernel of the base change map on Picard groups $\Pic(R/v^{n+1})\to \Pic(R/v^n)$ is contained in $\pi_0 \Sigma \Sigma^{n|v|}R/v = \pi_{-1-n|v|}R/v$. The latter is zero by assumption. This proves the injectivity as claimed. 
	\end{proof}
	\begin{exmp}
		
		The assumption $\pi_{-1-n|v|}(R/v)=0$ in \Cref{cor:pic_inj} cannot be removed, which we demonstrate with an example. Consider $R=\Sbb_{\K(1)}$ for an odd prime $p$. Then the maps $\Pic(\Sbb_{\K(1)}/p^{n+1})\to\Pic(\Sbb_{\K(1)}/p^n)$ are not injective, because the periodicity of the rings $\Sbb_{\K(1)}/p^{k+1}$ increases as $k$ increases, so there are suspensions of the unit in the kernel. However, $\pi_{-1}(\Sbb_{\K(1)}/p)=\Z/p\ne 0$, so indeed, the hypotheses of \Cref{cor:pic_inj} are not satisfied. In fact, the kernel of these maps on Picard groups is $\ZZ/p$ (see \Cref{thm:Pic_S0pk}), which maps isomorphically onto $\pi_{-1}(\Sbb_{\K(1)}/p)$, the upper bound on the relative Picard group coming from \Cref{thm:mod_quot}.
	\end{exmp}
	
	\subsection{Picard groups of regular quotients of complete local rings}
	
	The goal of this section is to use the tools above to compute Picard groups of certain ring spectra.
	
	\begin{defn}
		A regular local graded ring is a local Noetherian $\ZZ$-graded ring $R$ that contains a regular sequence of homogeneous elements $m_i \in \pi_*R$ such that $\pi_*R/(m_i)$ is a graded field.
	\end{defn}
	
	We now give a particular construction of quotients of $\pi_*$-regular local $\EE_2$-rings by powers of elements in a regular sequence.
	\begin{const}\label{construction:quotientregularsequence}
		Suppose $R$ be an $\EE_2$-ring such that $\pi_*R$ is concentrated in even degrees. Choose a sequence $m_1,\cdots,m_k$ of homogeneous elements of $\pi_*R$. In \cite[Section 2]{hahn2018quotients}, it is shown that there is an $\EE_2$-algebra map $\SP[x_1,\cdots,x_k] \to R$ taking $x_i$ to $m_i$, where $\SP[x_1,\cdots,x_k]$ is an augmented $\ZZ^k$-graded $\EE_2$-algebra and  each $x_i$ has multi-grading so that they form a basis. Fixing such a map, given a sequence of positive integers $a_1,\cdots,a_k$, we may form the $\EE_1$-$R$-algebra $R/(x_1^{a_1},\cdots,x_k^{a_k})$ as a relative tensor product $R\otimes_{\SP[x_1^{a_1},\cdots,x_k^{a_k}]}\SP$, where $\SP[x_1^{a_1},\cdots,x_k^{a_k}]$ is the algebra obtained by restricting $\SP[x_1,\cdots,x_k]$ to the multi-gradings indicated. 
	\end{const}
	\begin{lem}\label{lemma:evenquotcomparison}
		Suppose that $R$ is as in \Cref{construction:quotientregularsequence}. The tower of $\EE_1$-$R$-algebras \begin{equation*}
			R/x_1^{a_1}\otimes_R\cdots\otimes_R R/x_i^{\bullet}\otimes_R\cdots\otimes_R R/x_k^{a_k}
		\end{equation*} is isomorphic to the descent tower $D_{\bullet-1}(f)$, where $f$ is the map \[R/x_1^{a_1}\otimes_R\cdots\otimes_R \widehat{R/x_i}\otimes_R\cdots\otimes_R R/x_k^{a_k} \longrightarrow R/x_1^{a_1}\otimes_R\cdots\otimes_R  R/x_i\otimes_R\cdots\otimes_R  R/x_k^{a_k}.\]
		Here $\widehat{R/x_i}$ indicates that the tensor factor involving $R/x_i$ is missing in the source.
	\end{lem}
	
	\begin{proof}
		We first prove that the descent tower of the map of graded $\EE_1$-$\SP[x_i]$-algebras $g\colon\SP[x_i] \to \SP$ sending $x_i$ to zero is the tower $\SP[x_i]/x_i^{\bullet +1}$. There is a comparison map by observing that the descent tower $D_i(g)$ is concentrated in gradings $\leq i+1$, so it receives a map from $\SP[x_i]/x^{i+1}$, since it is obtained via the monoidal localization in graded $\SP[x_i]$-modules away from objects in degrees at least $i+1$. That the map is an equivalence follows from the description of the descent tower from \Cref{lemma:bimoduleofdescenttower}.
		
		Note that the descent tower of a map is preserved along monoidal exact functors. Since $f$ is the base change of the $\EE_1$-$\SP[x_i]$-algebra map $$R/x_1^{a_1}\otimes_R\cdots\otimes_R  R[x_i]\otimes_R\cdots\otimes_R  R/x_k^{a_k} \longrightarrow R/x_1^{a_1}\otimes_R\cdots\otimes_R  \widehat{R/x_i}\otimes_R\cdots\otimes_R  R/x_k^{a_k}$$ along the $\EE_2$-algebra map $\SP[x_i] \to \SP$ sending $x_i$ to $0$, the result follows.
	\end{proof}
	
	\begin{defn}
		Let $R$ be an $\EE_2$-ring and $\Ccal \subseteq \Perf(R)$ be a thick subcategory generated by a single object. Then the category $\Mod(R)^{\wedge}_\Ccal$ is defined to be the Bousfield localization of $\Mod(R)$ with respect to tensoring with some generator $M \in \Ccal$. We note that this is independent of the choice of generator. If $I\trianglelefteq \pi_*R$ is a finitely generated ideal, we use $\Mod(R)^{\wedge}_I$ to denote $\Mod(R)^{\wedge}_{\Ccal_I}$ where $\Ccal_I$ is the thick subcategory generated by $(R/x_1)\otimes_R(R/x_2)\otimes\cdots\otimes_R(R/x_n)$ where $I = (x_1,\cdots,x_n)$. We note this is independent of the choice of generators. 
	\end{defn}
	
	\begin{prop}\label{prop:compreglocalpic}
		Let $R$ be an $\EE_2$-ring such that $\pi_*R$ is even and a complete regular local graded ring with maximal ideal $\mfrak$. Then the map $\ZZ = \Pic(\SP) \to \Pic(\Mod(R))^{\wedge}_{\mfrak})$ is surjective.
	\end{prop}
	
	In particular, in the above situation, $\Pic(\Mod(R)_{\mathcal{m}})$ is cyclic of order depending only on the periodicity of $\pi_*R$. Note that this proposition generalizes \cite[Theorem 37]{Baker-Richter_invertible_modules} and  \cite[Theorem 2.4.6]{MS_Picard}.
	\begin{proof}
		Choose a regular sequence $(x_1,\cdots,x_n)$ generating $\pi_*R$, and consider the ring map $R \to R/(x_1,\cdots,x_n)$ coming from \Cref{construction:quotientregularsequence}. If $M \in \Pic(R)$, then $-\otimes_RM$ gives an automorphism of $\Mod(R/(x_1,\cdots,x_n))$. Since $\pi_*(R/(x_1,\cdots,x_n))$ is a field, every module is free, so this automorphism must send the unit to $\Sigma^nR/(x_1,\cdots,x_m)$ for some $n$. It now follows from running the Bockstein spectral sequences for each $x_i$ that $M \cong \Sigma^nR$.
	\end{proof}
	
	We next observe that the relative module space for the map $R/(x_1^{a_1},\cdots,x_m^{a_k}) \to R/(x_1,\cdots,x_m)$ is connected.
	
	\begin{lem}\label{lem:injectivity}
		Let $R$ be an $\EE_2$-ring such that $\pi_*R$ is even and let $x_1,\cdots,x_m$ be a sequence of elements in $\pi_*R$ such that $\pi_{-1}R/(x_1,\cdots,x_m)=0$. Then $\Perf(R/(x_1^{a_1},\cdots,x_m^{a_m})\mid  R/(x_1,\cdots,x_m))$ is connected.
	\end{lem}
	
	\begin{proof}
		We prove this by induction on $\sum_{i=1}^m a_i$, with the base case being $a_i=1$ for all $i$, where the result is trivial. So if we fix an $i$, because an extension of connected spaces is connected, in the inductive step, it is enough to show 
		$$\Perf(R/(x_1^{a_1},\cdots, x_i^{a_{i}+1},\cdots,x_m^{a_m})\mid  (R/(x_1^{a_1},\cdots, x_i^{a_{i}},\cdots,x_m^{a_m}))$$ is connected. But by \Cref{example:quotientbyanelement} and \Cref{lemma:evenquotcomparison}, this is the relative module space of the map $D_{i}(f) \to D_{i-1}(f)$ where $f$ is the map $A = R/(x_1^{a_1},\cdots,\widehat{x_i^{a_i}},\cdots,x_m^{a_m}) \to B = R/(x_1^{a_1},\cdots,x_m^{a_m})$. Thus by \Cref{thm:mod_quot}, it is enough to check that $\pi_{-1} I^{\otimes_A i+1}\otimes_AB=0$, where $I \to A \to B$ is a fiber sequence. But $I \cong A$ as an $A$-bimodule, so it is enough to see that $\pi_{-1}B = 0$. But this follows because $B$ is built under extensions by $R/(x_1,\cdots,x_m)$, which by assumption has $\pi_{-1}$ vanishing.
	\end{proof}

	In the regular local case, we can bound the essential image of $\Pic(R/(x_1^{a_1},\cdots,x_m^{a_m}))$ in $\Pic(R/(x_1,\cdots,x_m))$:
	
	\begin{thm}\label{iteratedquotient}
		Let $R$ be an $\EE_2$-ring such that $\pi_*R$ is even and a complete regular graded ring with maximal ideal $\mfrak$ generated by the regular sequence $x_1,\dots,x_m$. Then for any $\EE_2$-refinement of the $\EE_1$-$R$-algebra structure on $R/(x_1^{a_1},\cdots,x_m^{a_m})$ from \Cref{construction:quotientregularsequence}, $\Pic(R/(x_1^{a_1},\cdots,x_m^{a_m}))$ is generated by $\Sigma R$.
		
	\end{thm}
	
	\begin{proof}
		Let $A = R/(x_1^{a_1},\cdots,x_m^{a_m})$ and $B = R/(x_1,\cdots,x_m)$, which has homotopy ring a field.
		By \Cref{lem:injectivity}, it is enough to show that if $M \in \Pic(A)$, then $B\otimes_AM \cong \Sigma^nB$ for some $n$. We compute $\End_{B}(B\otimes_AM) = \mmap_{A}(M,B\otimes_AM) = \mmap_{A}(A,(B\otimes_AM)\otimes_AM^{\vee}) = B$. Since $\pi_*B$ is a field, and the endomorphism ring of $B\otimes_AM$ is $B$ as an $R$-module, it follows that $B\otimes_AM$ must be $\Sigma^nB$ for some $n$.
	\end{proof}
	
	\begin{rem}
		If $R$ is an $\EE_{2+m}$-algebra, if $a_i> 2+m-i$, then $R/(x_1^{a_1},\cdots,x_m^{a_m})$ admits an $\EE_2$-algebra structure by an iterated application of \cite{Burkland_mult_Moore}. 
	\end{rem}
	
	\begin{exmp}\label{exm:quotientsofEn}
		Let $\kappa$ be a perfect field of characteristic $p$ and $G$ be a formal group of height $n\geq 1$ over $\kappa$. Then let $\E(\kappa,G)$ denote the Lubin--Tate theory associated to this data, so that $\pi_*\E(\kappa,G) \cong \mathrm{W}(\kappa)\llb u_1,\cdots,u_{n-1}\rrb [u^{\pm1}]$. The ring $\pi_*\E(\kappa,G)$ is even and complete regular local, with an example of a regular sequence being $(p,v_1,\cdots,v_{n-1})$. It follows from \Cref{prop:compreglocalpic} that $\Pic_{\K(n)}\Mod(\E(\kappa,G)) = \Z/2$, and it follows from \Cref{iteratedquotient} that $\Pic(\E(\kappa,G)/(p^{a_0},\cdots,v_{n-1}^{a_{n-1}})) = \Z/2$ when the $\EE_1$-structure of the quotient ring $\E(\kappa,G)/(p^{a_0},\cdots,v_{n-1}^{a_{n-1}})$ refines to an $\EE_2$-structure.
	\end{exmp}
	\section{Computations of Picard groups of quotient ring spectra}\label{section:computations}
	In this section, we compute Picard groups of certain finite spectra. As first examples, we compute Picard groups of the quotients of Tate $\K$-theories and the quotients $\KO/\eta^k$ following the proof of \Cref{thm:mod_quot}. Then we turn to Picard groups of $\K(n)$-local generalized Moore algebras. The key computation tool is the (profinite) descent spectral sequence for Picard groups of $\K(n)$-local ring spectra.
	\begin{thm}[{\cite{LZ_profin_Picard,mor2023picard,Heard_2021Sp_kn-local,MS_Picard}}]\label{thm:picss}
		Let $M$ be a finite spectra that admits an $\Ebb_2$-ring structure. For any closed subgroup $G$ of the Morava stabilizer group $\Gbb_n$, we have a descent spectral sequence of Picard groups whose $E_2$-page is
		\begin{equation}\label{eqn:PicSS_M}
			^{\Pic} E_2^{s,t}=\H_c^s(G;\pi_t\cPic_{\K(n)}(\E_n\otimes M))\Longrightarrow \pi_{t-s}\cPic_{\K(n)}(\E_n^{hG}\otimes M), \qquad t-s\ge 0.
		\end{equation}
	\end{thm} 
	
	In the case when $M$ is a generalized Moore $\EE_2$-algebra of type $n$ as in \cite{Burkland_mult_Moore}, \Cref{exm:quotientsofEn} implies that $\Pic(\E_n\otimes M)=\Z/2$.
	\subsection{Picard groups of quotients of $\KO$ and Tate $\K$-theories}
	
	The goal of this subsection is to compute Picard groups of the following quotients using the tools of \Cref{section:tools}.
	
	\begin{defn}\label{dfn:tateandkoquot}
		We define some $\EE_{\infty}$-rings of interest.
		\begin{enumerate}
			\item Wood's cofiber sequence
			\begin{equation*}
				\begin{tikzcd}
					\Sigma\KO\rar["\eta"]&\KO\rar& \KU
				\end{tikzcd}
			\end{equation*}
			implies that $\KU\simeq \KO/\eta$ as $\einf$-rings. Then $\KO/\eta^k$ obtains an $\EE_{\infty}$-structure, realized inductively as the pullback $\KO/\eta^{k-1}\times_{\KO/\eta^{k-1}\otimes_{\KO}\KU}\KU$.  Note that the underlying $\EE_1$-algebras agree with the descent tower of \Cref{definition:E1descenttower}, where the pullback is denoted as $D_k(\KO\to \KU)$.
			\item Consider the real and complex Tate $\K$-theory $\KO\llb q\rrb$ and $\KU\llb q\rrb$. Notice $\KO\simeq \KO\llb q\rrb /q$ and $\KU\simeq \KU\llb q\rrb /q$ as $\einf$-rings, we similarly obtain $\EE_{\infty}$-structures on $\KO\llb q\rrb /q^k$ and $\KU\llb q\rrb /q^k$ for any $k$.
		\end{enumerate}	
		
	\end{defn}
	An immediate consequence of \Cref{cor:pic_inj} is Picard groups of $\KU\llb q\rrb/q^k$ and $\KO\llb q\rrb/q^k$. 
	\begin{prop}\label{prop:Pic_TateK}
		For any $k\ge 1$, we have $\Pic(\KU\llb q\rrb/q^k)=\Z/2$ and $\Pic(\KO\llb q\rrb/q^k)=\Z/8$. 
	\end{prop}
	\begin{proof}
		By \Cref{thm:mod_quot}, we have the base change maps on Picard groups $\Pic(\KU\llb q\rrb/q^{k+1})\to \Pic(\KU\llb q\rrb/q^{k})$ and $\Pic(\KO\llb q\rrb/q^{k+1})\to \Pic(\KO\llb q\rrb/q^{k})$ are injective for any $k$, since $\pi_{-1}(\KU\llb q\rrb/q)=\pi_{-1}(\KU)=0$ and $\pi_{-1}(\KO\llb q\rrb/q)=\pi_{-1}(\KO)=0$, respectively. It follows that the base change maps \begin{equation*}
			\Pic(\KU\llb q\rrb/q^{k})\to \Pic(\KU)=\Z/2,\qquad \Pic(\KO\llb q\rrb/q^{k+1})\to \Pic(\KO)=\Z/8,
		\end{equation*} are injections, where $\Pic(\KU)$ and $\Pic(\KO)$ are computed in \cite{Gepner-Lawson_Brauer,MS_Picard,Baker-Richter_invertible_modules}.   The two maps are both isomorphism since homotopy groups of $\KU\llb q\rrb/q^{k}$ and $\KO\llb q\rrb/q^{k}$ are $2$ and $8$-periodic, respectively. 
	\end{proof}
	Taking the inverse limit as $k\to \infty$, we obtain a completed version of the computation from \cite{MS_Picard}:
	\begin{cor}\label{cortatek}
		The Picard groups $\Pic( {\Mod}(\KU\llb q\rrb)^{\wedge}_q)$ and $\Pic( {\Mod}(\KO\llb q\rrb)^{\wedge}_q)$ of  the $q$-complete module categories of $\KU\llb q\rrb$ and $\KO\llb q\rrb$ are isomorphic to $\Z/2$ and $\Z/8$, respectively. 
	\end{cor}
	\begin{proof}
		The proof of \cite[Main Theorem B]{LZ_profin_Picard} by the second and the third authors shows that
		\begin{align*}
			\Pic( {\Mod}(\KU\llb q\rrb)^{\wedge}_q)& \simto \lim_k \Pic(\KU\llb q\rrb/q^{k})\cong \Z/2,\\
			\Pic( {\Mod}(\KO\llb q\rrb)^{\wedge}_q)& \simto \lim_k \Pic(\KO\llb q\rrb/q^{k})\cong \Z/8.\qedhere
		\end{align*}
	\end{proof}
	\begin{rem}
		We note that it is possible to use \cite[Proposition 10.11]{Mathew_Galois} to show that the uncompleted and completed categories of modules of the complex Tate $\K$-theory $\KU\llb q\rrb$ in \Cref{cortatek} have the same Picard groups, giving another proof of \cite[Proposition 7.2.7]{MS_Picard}.
	\end{rem}
	We next turn to computing Picard groups of quotients of $\KO$ by powers of $\eta\in \pi_1(\KO)$.
	\begin{thm}
		The Picard groups of $\KO/\eta^k$ for $2\le k\le 5$ are:
		\begin{equation*}
			\Pic(\KO/\eta^k)=\left\{\begin{array}{@{}lcc}
				\Z/4, &k=2,3;& \textup{(\Cref{thm:PicKOeta2})}\\
				\Z/8, & k=4,5.& \textup{(\Cref{thm:eta4})}
			\end{array}\right.
		\end{equation*}
	\end{thm}
	\begin{lem}\label{lem:KOeta2}
		When $n\ge 3$, $\KO/\eta^n\simeq \KO\oplus \Sigma^{n+1}\KO$ as a $\KO$-module spectrum. For $n=2$, $\KO/\eta^2$ contains a unit in $\pi_4$, and its homotopy groups are:
		\begin{equation*}
			\pi_k(\KO/\eta^2)=\begin{cases}
				\Z,&k\equiv 0,3\mod 4;\\
				\Z/2&k\equiv 1\mod 4;\\
				0& k\equiv 2\mod 4. 
			\end{cases}
		\end{equation*}
	\end{lem}
	\begin{proof}
		When $n\ge3$, we have $\eta^n=0\in \pi_n(\KO)$. Therefore the cofiber of $\Sigma^{n}\KO\xrightarrow{\eta^n=0}\KO$ is equivalent to $\KO\oplus \Sigma^{n+1}\KO$ as a $\KO$-module. 
		
		In the $n=2$ case, first recall $\KU\simeq\KO/\eta$ by Wood's cofiber sequence. The $\EE_{\infty}$-structure on $\KO/\eta^2$ is obtained via the pullback square:
		\begin{equation}\label{eqn:KOeta2}
			\begin{tikzcd}
				\KO/\eta^2\rar\dar\ar[dr,phantom,"\lrcorner", very near start]& \KU\dar\\
				\KU\rar& \KU\otimes_{\KO}\KU.
			\end{tikzcd}
		\end{equation}
		As $\KO\to\KU$ is a $C_2$-Galois extension, the lower right corner of this pullback diagram is equivalent to the $\einf$-ring spectrum $\map(C_2,\KU)$. Under this identification, 
		\begin{itemize}
			\item the left unit map $\KU\xrightarrow{\id\otimes 1} \KU\otimes_\KO\KU\simeq \map(C_2,\KU)$ is adjoint to the trivial $C_2$-action on $ \KU$;
			\item the right unit map $\KU\xrightarrow{1\otimes \id} \KU\otimes_\KO\KU\simeq \map(C_2,\KU)$ is adjoint to the $C_2$-action on $\KU=\K\Rbb$. 
		\end{itemize}
		The long exact sequence in homotopy groups associated to the pullback diagram \eqref{eqn:KOeta2} then splits into shorter ones of the forms:
		\begin{align*}
			0\to \pi_{4m}(\KO/\eta^2)\to \pi_{4m}(\KU)\oplus\pi_{4m}(\KU)&\xrightarrow{\begin{pmatrix}
					1&1\\1&1
			\end{pmatrix}}\pi_{4m}(\KU)\oplus\pi_{4m}(\KU)\to  \pi_{4m-1}(\KO/\eta^2)\to 0,\\
			0\to \pi_{4m+2}(\KO/\eta^2)\to \pi_{4m+2}(\KU)\oplus\pi_{4m+2}(\KU)&\xrightarrow{\begin{pmatrix}
					1&1\\1&-1
			\end{pmatrix}}\pi_{4m+2}(\KU)\oplus\pi_{4m+2}(\KU)\to  \pi_{4m+1}(\KO/\eta^2)\to 0.
		\end{align*}
		The claim follows by computing the kernels and the cokernels of those maps. Note we in particular see that the unit in $\pi_4$ of $\KU$ lifts to $\KU/\eta^2$.
	\end{proof}
	\begin{rem}
		Alternatively, one might try to compute $\pi_*(\KO/\eta^2)$ using the long exact sequence associated to the cofiber sequence $\Sigma^2\KO\xrightarrow{\eta^2}\KO\to \KO/\eta^2$. While this method completely determines $\pi_k(\KO/\eta^2)$ when $k\not\equiv 4\mod 8$, it only computes $\pi_{8m+4}(\KO/\eta^2)$ up to an extension.
	\end{rem}
	Under the $\einf$-ring structures given by the pullback squares,  one can check that the ring structures on the homotopy groups of $\KO/\eta^2$ and $\KO/\eta^3$ are:
	\begin{align*}
		\pi_*(\KO/\eta^2)&=\Z[\eta,\zeta,u^{\pm 2}]/(2\eta,\eta^2,\zeta\eta,\zeta^2), &&|\eta|=1,|\zeta|=-1, |u^2|=4.\\
		\pi_*(\KO/\eta^3)&=\Z[\eta,\varepsilon,u^{\pm 2}]/(2\eta,\eta^3,\varepsilon\eta,\varepsilon^2), &&|\eta|=1,|\varepsilon|=0, |u^2|=4.
	\end{align*}
	In particular, this multiplicative structure on $\KO/\eta^3$ is not equivalent to a square-zero extension of $\KO$  as an $\mathbb{A}_2$-$\KO$-algebra.
	\begin{rem}
		Here is an alternative explanation as to why $\KO/\eta^2$ and $\KO/\eta^3$ are $4$-periodic. Recall in the HFPSS of $\KO\simeq \KU^{hC_2}$, there is a $d_3$-differential $d_3(u^2)=\eta^3$. The target of this differential is killed in the quotient rings $\KO/\eta^2$ and $\KO/\eta^3$. This makes $u^2$ a permanent cycle in the corresponding HFPSS, which eventually becomes the periodicity element in the homotopy rings $\pi_*(\KO/\eta^2)$ and $\pi_*(\KO/\eta^3)$. By building our quotients in filtered $\EE_{\infty}$-rings, it is then possible to directly see that $u^2$ lifts to $\KO/\eta^3$.
	\end{rem}
	\begin{rem}
		From the computation above, we can see the base change map on Picard groups $\Pic(\KO/\eta^2)\to \Pic(\KO/\eta)=\Pic(\KU)=\Z/2$ cannot be injective, since the homotopy groups $\pi_*(\KO/\eta^2)$ are not even periodic. Notice $\pi_{-|\eta|-1}(\KO/\eta)=\pi_{-2}(\KU)\ne 0$,  this gives another example of the necessity of the assumption $\pi_{-n|v|-1}(R/v)=0$ for $\Pic(R/v^{n+1})\to \Pic(R/v^n)$ to be injective in \Cref{cor:pic_inj}.
	\end{rem}
	\begin{prop}\label{thm:PicKOeta2}
		$\Pic(\KO/\eta^3)\simto \Pic(\KO/\eta^2)=\Z/4$. 
	\end{prop}
	\begin{proof}
		The ring structure on $\KO/\eta^2$ is defined via the pullback square \eqref{eqn:KOeta2} of $\einf$-ring spectra. Applying the oriented pullback  \Cref{const:pullbackbimodule}, we obtain a diagram of categories:
		\begin{equation*}
			\begin{tikzcd}
				\Mod(\KO/\eta^2)\ar[dr,dashed,start anchor=south east]\ar[ddr,bend right=15,end anchor=west]\ar[drr,bend left=15,end anchor=north]&&\\&\mathcal{PB}\rar\dar\ar[dr,phantom,"\lrcorner", very near start]& \Mod(\KU)\dar\\
				&\Mod(\KU)\rar& \Mod(\map(C_2,\KU)),
			\end{tikzcd}
		\end{equation*}
		where $\Mod(\KO/\eta^2)\to \mathcal{PB}$ is a fully faithful embedding of symmetric monoidal $\infty$-categories.  By \cite[Proposition 2.2.3]{MS_Picard}, taking Picard spaces commute with limits of monoidal categories. This yields a pullback square of group-like $\einf$-spaces:
		\begin{equation*}
			\begin{tikzcd}
				{\cPic}(\mathcal{PB})\rar\dar\ar[dr,phantom,"\lrcorner", very near start]& {\cPic}(\KU)\dar\\
				{\cPic}(\KU)\rar& {\cPic}(\map(C_2,\KU)),
			\end{tikzcd}
		\end{equation*}
		where 
		\begin{itemize}
			\item $\Pic(\KO/\eta^2)\to \Pic(\mathcal{PB})=\pi_0(\cPic(\mathcal{PB}))$ is injective;
			\item $\pi_k\cPic(\mathcal{PB})\cong \pi_k(\cPic(\KO/\eta^2))$ for all $k\ge 1$; In particular, $\pi_1(\cPic(\mathcal{PB}))=\pi_0(\KO/\eta^2)^\times= \Z/2$. 
			\item $\pi_k\cPic(\map(C_2,\KU))\cong \pi_k\map(C_2,\cPic(\KU))\cong \map(C_2,\pi_k\cPic(\KU))$ for all $k\ge 0$.
		\end{itemize}
		From this, we obtain a long exact sequence in homotopy groups of Picard spaces: 
		\begin{align}
			\map(C_2,\pi_1(\KU))\longrightarrow &\pi_0(\KO/\eta^2)^\times\longrightarrow \pi_0(\KU)^\times\oplus\pi_0(\KU)^\times\longrightarrow \map(C_2,\pi_0(\KU)^\times)\nonumber\\
			\longrightarrow& \Pic(\mathcal{PB})\longrightarrow \Pic(\KU)\oplus \Pic(\KU)\longrightarrow \map(C_2,\Pic(\KU)).\label{eqn:picKOeta2}
		\end{align}
		Plugging in the known homotopy groups and maps between them as in the proof of \Cref{lem:KOeta2}, the Picard group $\Pic(\mathcal{PB})$ sits in an exact sequence:
		\begin{equation*}
			0\longrightarrow \Z/2\longrightarrow \Z/2\oplus \Z/2\xrightarrow{\begin{pmatrix}
					1&1\\1&1
			\end{pmatrix}} \Z/2\oplus\Z/2\longrightarrow \Pic(\mathcal{PB})\longrightarrow \Z/2\oplus\Z/2\xrightarrow{\begin{pmatrix}
					1&1\\1&1
			\end{pmatrix}}\Z/2\oplus \Z/2.
		\end{equation*}
		This determines $\Pic(\mathcal{PB})$ up to an extension:
		\begin{equation*}
			0\to \Z/2\longrightarrow \Pic(\mathcal{PB})\longrightarrow \Z/2\to 0.
		\end{equation*}
		From the computation of $\pi_*(\KO/\eta^2)$ in \Cref{lem:KOeta2}, we know $\KO/\eta^2$ is $4$-periodic. Therefore $\Pic(\mathcal{PB})\supseteq \Pic(\KO/\eta^2)$ contains a cyclic subgroup of order $4$. Combined with the extension problem above, this forces $ \Pic(\KO/\eta^2)=\Pic(\mathcal{PB})=\Z/4$. 
		
		By \Cref{cor:pic_inj}, the base change map $\Pic(\KO/\eta^3)\to \Pic(\KO/\eta^2)=\Z/4$ is injective since $\pi_{-2|\eta|-1}(\KO/\eta)=\pi_{-3}(\KU)=0$. As a result, $\Pic(\KO/\eta^3)$ is either $0$, $\Z/2$, or $\Z/4$. By \Cref{lem:KOeta2}, homotopy groups of $\KO/\eta^3\simeq \KO\oplus\Sigma^4\KO$ (as $\KO$-module spectra) are not even periodic, yielding $\Pic(\KO/\eta^3)=\Z/4$.  
	\end{proof}
	
	\begin{rem}
		We note that the proof above fails for the quotients $\KO^\wedge_2/\eta^2$ and $\KO/(2^m,\eta^2)$. This is because $\pi_0(\KU^\wedge_2)^\times=\Z_2^\times$ and $\pi_0(\KU/2^m)^\times=(\Z/2^m)^\times$ are larger than $\pi_0(\KU)^\times=\Z^\times=\{\pm 1\}$. 
		
		Rerunning the long exact sequence \eqref{eqn:KOeta2} after $2$-completing all the spectra inside and using the periodicity of $\pi_*(\KO^\wedge_2/\eta^2)$, one can show that $\Pic(\mathcal{PB}^\wedge_2)=\Z_2\oplus\Z/4$. However, it is not clear whether $\Pic(\KO^\wedge_2/\eta^2)$ is a proper subgroup of $\Pic(\mathcal{PB}^\wedge_2)$ or not. 
	\end{rem}
	\begin{prop}\label{thm:eta4}
		$\Pic(\KO/\eta^5)\simto \Pic(\KO/\eta^4)=\Z/8$. 
	\end{prop}
	\begin{proof}
		We first show $\Pic(\KO/\eta^4)=\Z/8$ by applying the method in \Cref{thm:PicKOeta2} to the pullback square of $\einf$-ring spectra:
		\begin{equation*}
			\begin{tikzcd}
				\KO/\eta^4\rar\dar\ar[dr,phantom,description,"\lrcorner", very near start] & \KO/\eta^2\dar\\
				\KO/\eta^2\rar& \KO/\eta^2\otimes_{\KO}\KO/\eta^2. 
			\end{tikzcd}
		\end{equation*}
		This induces a pullback square of group like $\einf$-spaces by \cite[Lemma 1.7]{LT_pullbacks} and \cite[Proposition 2.2.3]{MS_Picard}:
		\begin{equation*}
			\begin{tikzcd}
				{\cPic}(\mathcal{PB})\rar\dar\ar[dr,phantom,description,"\lrcorner", very near start] &{\cPic}( \KO/\eta^2)\dar\\
				{\cPic}(\KO/\eta^2)\rar&{\cPic}( \KO/\eta^2\otimes_{\KO}\KO/\eta^2), 
			\end{tikzcd}
		\end{equation*}
		where the natural map $\cPic(\KO/\eta^4)\to \cPic(\mathcal{PB})$ is a fully faithful embedding of $\infty$-groupoids. This induces a long exact sequence in homotopy groups:
		\begin{align}
			\cdots \to \pi_1( \KO/\eta^2\otimes_{\KO}\KO/\eta^2)&\to \pi_0(\KO/\eta^4)^\times \to \pi_0(\KO/\eta^2)^\times \oplus\pi_0(\KO/\eta^2)^\times \to \pi_0( \KO/\eta^2\otimes_{\KO}\KO/\eta^2)^\times\nonumber\\ &\to \Pic(\mathcal{PB})\to \Pic(\KO/\eta^2)\oplus\Pic(\KO/\eta^2)\to \Pic( \KO/\eta^2\otimes_{\KO}\KO/\eta^2). \label{eqn:LES_picKOeta4}
		\end{align}
		Let us analyze this long exact sequence:
		\begin{itemize}
			\item  By \Cref{thm:PicKOeta2}, we have $\Pic(\KO/\eta^2)\cong \Z/4$ is generated by $\Sigma \KO/\eta^2$. It follows that $\ker( \Pic(\KO/\eta^2)\oplus\Pic(\KO/\eta^2)\to \Pic( \KO/\eta^2\otimes_{\KO}\KO/\eta^2))\cong \Z/4$ is the diagonal of the source. 
			\item 	Note $\KO/\eta^2\otimes_{\KO}\KO/\eta^2\simeq \KO/\eta^2\oplus \Sigma^3 \KO/\eta^2$ (as $\KO$-module spectra). By \Cref{lem:KOeta2}, we have 
			\begin{equation*}
				\pi_0( \KO/\eta^2\otimes_{\KO}\KO/\eta^2)=\Z\oplus\Z/2, \qquad \pi_1( \KO/\eta^2\otimes_{\KO}\KO/\eta^2)=\Z/2. 
			\end{equation*}
			Note the left unit map $\KO/\eta^2\to \KO/\eta^2\otimes_{\KO}\KO/\eta^2$ induces an injective ring homomorphism $\Z=\pi_0(\KO)\to \pi_0( \KO/\eta^2\otimes_{\KO}\KO/\eta^2)=\Z\oplus \Z/2$. This forces the ring structure on $ \pi_0( \KO/\eta^2\otimes_{\KO}\KO/\eta^2)$ to be a square-zero extension of $\Z$ by the abelian group $\Z/2$. We thus obtain 
			\begin{equation*}
				\pi_0( \KO/\eta^2\otimes_{\KO}\KO/\eta^2)^\times \cong \Z/2\oplus\Z/2. 
			\end{equation*}
			\item Note $\Z=\pi_0(\KO/\eta^4)\to \pi_0(\KO/\eta^2)=\Z$ is a ring map, which is necessarily the identity. We have the map $\Z/2=\pi_0(\KO/\eta^4)^\times \to \pi_0(\KO/\eta^2)^\times \oplus\pi_0(\KO/\eta^2)^\times=\Z/2\oplus\Z/2$ is an injection. 
		\end{itemize}
		Based on the observations above, the long exact sequence becomes:
		\begin{equation*}
			\pi_1( \KO/\eta^2\otimes_{\KO}\KO/\eta^2)\xrightarrow{0} \Z/2\to \Z/2\oplus\Z/2\to \Z/2\oplus\Z/2\to \Pic(\mathcal{PB})\to \Z/4\oplus\Z/4\to  \Pic( \KO/\eta^2\otimes_{\KO}\KO/\eta^2).
		\end{equation*}
		This implies $\Pic(\mathcal{PB})$ sits in a short exact sequence:
		\begin{equation*}
			0\to \Z/2\longrightarrow \Pic(\mathcal{PB})\longrightarrow \Z/4\to 0. 
		\end{equation*}
		By \Cref{lem:KOeta2}, $\KO/\eta^4$ is at least $8$-periodic. This implies  $\Pic(\mathcal{PB})\supseteq \Pic(\KO/\eta^4)$ contains a cyclic group of order $8$. The short exact sequence above therefore forces $\Pic(\mathcal{PB})= \Pic(\KO/\eta^4)=\Z/8$. 
		
		By \Cref{cor:pic_inj}, the base change map $\Pic(\KO/\eta^5)\to \Pic(\KO/\eta^4)$ is injective, since $\pi_{-1-4|\eta|}(\KO/\eta)=\pi_{-5}(\KU)=0$. By \Cref{lem:KOeta2}, $\pi_*(\KO/\eta^5)=\pi_*(\KO)\oplus\pi_*(\Sigma^6\KO)$ is not $4$-periodic. This implies \begin{equation*}
			\Pic(\KO/\eta^5)\simto \Pic(\KO/\eta^4)=\Z/8.\qedhere
		\end{equation*} 
	\end{proof}
	\subsection{Picard groups of generalized Moore spectra of type $n$}
	At general heights, we have the following qualitative results of Picard groups of $\K(n)$-local generalized Moore spectra. 
	\begin{thm}\label{thm:Moore_Pic}
		Let $M=\mathbb{S}/(p^{d_0},v_1^{d_1},\cdots, v_{n-1}^{d_{n-1}})$ be a generalized Moore $\Ebb_2$-algebra, i.e. a generalized Moore spectrum of type $n$ equipped with an $\Ebb_2$-ring structure as in \cite{Burkland_mult_Moore}.
		\begin{enumerate}
			\item The Picard group $\Pic(\L_{\K(n)}M)$ is finite.
			\item When $2p-1> n^2$ and $(p-1)\nmid n$, the Picard group $\Pic(\L_{\K(n)}M)$ is algebraic in the sense of a short exact sequence:
			\begin{equation*}
				\begin{tikzcd}
					0\to \H^1_c(\Gbb_n;\pi_0(\E_n\otimes  M)^\times)\rar & \Pic(\L_{\K(n)}M)\rar &\Z/2 \to 0.
				\end{tikzcd}
			\end{equation*}
		\end{enumerate}
	\end{thm}
	\begin{proof}
		\begin{enumerate}
			\item The Picard group of $\L_{\K(n)}M$ can be computed by the profinite descent spectral sequence in \eqref{eqn:PicSS_M}. The argument in  \cite[\S3.4]{LZ_profin_Picard} by second and the third authors shows that the filtration on $\Pic(\L_{\K(n)}M)=\pi_0\cPic(\L_{\K(n)}M)$ from the $E_\infty$-page of \eqref{eqn:PicSS_M} is finite. The associated graded terms of this finite filtration are $^{\Pic} E_\infty^{s,s}$ for $0\le s\le N$ for some $N$, which are sub-quotients of 
			\begin{equation*}
				^{\Pic} E_2^{s,s}=\H_c^s(\Gbb_n;\pi_t\cPic(\E_n\otimes M))=\begin{cases}
					\H_c^0(\Gbb_n;\Pic(\E_n\otimes M)), & s=0;\\
					\H_c^1(\Gbb_n;\pi_0(\E_n\otimes M)^\times), &s=1;\\
					\H_c^s(\Gbb_n;\pi_{s-1}(\E_n\otimes M)), & s>1.
				\end{cases}
			\end{equation*}
			As a result, it suffices to prove the cohomology groups above are all finite. 
			\begin{itemize}
				\item For $s=0$, we have proved $\Pic(\E_n\otimes M)=\Z/2$ in \Cref{exm:quotientsofEn}. This implies $\H_c^0(\Gbb_n;\Pic(\E_n\otimes M))=\H_c^0(\Gbb_n;\Z/2)=\Z/2$ is finite. 
				\item For $s=1$, recall $\H_c^1(\Gbb_n;\pi_0(\E_n\otimes M)^\times)$  is represented by continuous crossed homomorphisms from $\Gbb_n$ to $\pi_0(\E_n\otimes M)^\times$. The set of such continuous crossed homomorphisms is finite since $\Gbb_n$ is topologically finitely generated and $\pi_0(\E_n\otimes M)^\times$ is a finite group.
				\item For $s> 1$, the finiteness of the continuous group cohomology $\H_c^s(\Gbb_n;\pi_{s-1}(\E_n\otimes M))$ was shown in the proof of \cite[Lemma 4.21]{fixedpt} (bottom of page 22). 
			\end{itemize}
			This proves the finiteness of $\Pic(\L_{\K(n)} M)$. 
			\item Recall that the $E_2$-page of the homotopy fixed point spectral sequence 
			\begin{equation*}
				E_2^{s,t}=\H_c^s(\Gbb_n;\pi_t(\E_n))\Longrightarrow \pi_{t-s}\left(\Sbb_{\K(n)}\right)
			\end{equation*}
			is sparse in the sense that $\H_c^s(\Gbb_n;\pi_t(\E_n))=0$ unless $2(p-1)\mid t$ (see \cite[Remark 1.4]{Goerss-Hopkins}). It follows that for an invariant ideal $\Ical=(p^{d_0},u_1^{d_1},\cdots, u_{n-1}^{d_{n-1}})\trianglelefteq \pi_0(\E_n)$, we have 
			\begin{equation*}
				\H_c^s(\Gbb_n;\pi_t(\E_n)/\Ical)=0 \text{ unless } 2(p-1)\mid t. 
			\end{equation*}
			This means in the descent filtration for $\Pic(\L_{\K(n)}M)$ from the spectral sequence \eqref{eqn:PicSS_M}, the next possible non-zero associated graded term after $E_\infty^{0,0}$ and $E_\infty^{1,1}$ is $E_{\infty}^{2p-1,2p-1}$, which is a sub-quotient of $E_2^{2p-1,2p-1}=\H_c^{2p-1}(\Gbb_n;\pi_{2p-2}(\E_n\otimes M))$. When $(p-1)\nmid n$, we have $\cd_p(\mathbb{S}_n)=n^2$. This imposes a horizontal vanishing line at $s=n^2$ on the $E_2$-page of \eqref{eqn:PicSS_M}. If we further assume $2p-1>n^2$, then  $E_{2}^{0,0}$ and $E_2^{1,1}$ are both permanent cycles and the only two non-zero terms in the $t-s=0$ stem on the $E_2$-page.  Therefore $\Pic(\L_{\K(n)}M)$ sits in an extension:
			\begin{equation*}
				\begin{tikzcd}
					0\to E_\infty^{1,1}=E_2^{1,1}=\H^1_c(\Gbb_n;\pi_0(\E_n\otimes  M)^\times)\rar & \Pic(\L_{\K(n)}M)\rar &\Z/2=E_2^{0,0}=E_\infty^{0,0}\to 0.
				\end{tikzcd}\qedhere
			\end{equation*}
		\end{enumerate}
	\end{proof}
	As an application to \Cref{thm:Moore_Pic}, we will compute Picard groups of $\K(1)$-local Moore spectra $\Sbb_{\K(1)}/p^k$ in \Cref{subsec:ht1}. At higher heights, Barthel--Schlank--Stapleton--Weinstein computed the even algebraic $\K(n)$-local Picard group $\Pic_{\K(n)}^{\mathrm{alg},0}\cong \H^1_c(\Gbb_n;\pi_0(\E_n)^\times)$ in \cite{BSSW_Picard}. When  $2p-1>n^2$ and $(p-1)\nmid n$, the Picard group of the $\K(n)$-local category is algebraic in the same sense as in \Cref{thm:Moore_Pic}. At height $n\ge 2$, their computation then yields \cite[Corollary 2.4.4]{BSSW_Picard}
	\begin{equation}\label{eqn:Pic_Kn}
		\Pic\left(\Sp_{\K(n)}\right)\cong \Zp\oplus\Zp\oplus \Z/(2p^n-2),
	\end{equation}
	generated by $\Sigma\Sbb_{\K(n)}$ and $\Sbb_{\K(n)}\ldetr^{\otimes(p-1)}$, where $\Sbb_{\K(n)}\ldetr$ is the determinant twist of the $\K(n)$-local sphere constructed in \cite{BBGS_det}. 
	
	In \cite{LZ_profin_Picard}, the second and the third authors used Burklund's result in \cite{Burkland_mult_Moore} to refine a construction in \cite{Hovey-Strickland_1999} that there is a tower of generalized Moore spectra $\{M_j\}$ of type $n$ such that $M_j$ is an $\Ebb_j$-algebra over $M_{j+1}$ and there is an equivalence $\Sbb_{\K(n)}\simto \lim_j \L_{\K(n)}M_j$ in $\CAlg_{\K(n)}$. Then they proved that 
	\begin{equation}\label{eqn:PicKn_lim}
		\Pic\left(\Sp_{\K(n)}\right) \simto \lim_j \Pic(\L_{\K(n)}M_j). 
	\end{equation}
	In light of \Cref{thm:Moore_Pic} and the computation \eqref{eqn:Pic_Kn} in \cite{BSSW_Picard}, it is a natural question to ask what is the even algebraic Picard group $\H^1_c(\Gbb_n;\pi_0(\E_n\otimes M)^\times)$ of a $\K(n)$-local generalized Moore $\Ebb_2$-algebra  $M$. Let's first analyze the images of the two generators from \eqref{eqn:Pic_Kn}. 
	\begin{itemize}
		\item (Regular spheres) By the Periodicity Theorem \cite{HS_nilp2}, $\L_{\K(n)}M$ admits a (minimal) periodicity $v_n^{p^N}$ for some $N\ge 0$.
		\item (Determinant spheres)  One can check that the determinant map $\Gbb_n\xrightarrow{\det} \Zpx\le \pi_0(\E_n)^{\times}$ has order $(p-1)p^{d_0-1}$ modulo $p^{d_0}$. 
	\end{itemize}
	Therefore, if $M=\Sbb/(p^{d_0},\cdots, v_{n-1}^{d_{n-1}})$, then the images of the two generators $\Sigma\Sbb_{\K(n)}$ and $\Sbb_{\K(n)}\ldetr^{\otimes(p-1)}$ of $\Pic(\Sp_{\K(n)})$ in $\Pic(\L_{\K(n)}M)$ have orders $2(p^n-1)p^N$ and $p^{d_0-1}$, respectively. Note that $d_0>1$ -- otherwise $M$ does not admit an $\Ebb_2$-ring structure as a consequence of the Hopkins--Mahowald theorem \cite[Theorem 4.18]{Mathew_2015}. Then  $\Sigma\Sbb_{\K(n)}$ and $\Sbb_{\K(n)}\ldetr^{\otimes(p-1)}$ are linearly independent over $M$. Hence when $n\ge 2$, $2p-1>n^2$, and $(p-1)\nmid n$, we obtain:
	\begin{equation*}
		\mathrm{Im}\left[\Pic\left(\Sp_{\K(n)}\right) \longrightarrow \Pic(\L_{\K(n)}M)\right]	=\Z/p^N\oplus\Z/p^{d_0-1}\oplus\Z/(2p^n-2). 
	\end{equation*}
	\begin{quest}\label{quest:Pic_M}
		Under the assumptions above, is the map $\Pic\left(\Sp_{\K(n)}\right) \longrightarrow \Pic(\L_{\K(n)}M)$ surjective so that $ \Pic(\L_{\K(n)}M)\cong \Z/p^N\oplus\Z/p^{d_0-1}\oplus\Z/(2p^n-2)$? 
	\end{quest}
	Note that the inverse limit of the images agree with $\Pic\left(\Sp_{\K(n)}\right)$ in \eqref{eqn:Pic_Kn}. This is true more generally by the proof of \cite[Proposition 14.3.(d)]{Hovey-Strickland_1999}. Therefore anything in the complement of the image in \Cref{quest:Pic_M} does not survive the inverse limit process in  \eqref{eqn:PicKn_lim}. 
	\subsection{Computations at height $1$}\label{subsec:ht1}
	By \cite{Burkland_mult_Moore}, $\mathbb{S}/p^k$ admits an $\Ebb_2$-ring structure when $p>2$ and $k\ge 3$. We apply \Cref{thm:picss} to the Picard groups of the $\K(1)$-localizations to obtain:
	\begin{thm}\label{thm:Pic_S0pk}
		When $p>2$ and $k\ge 3$, we have  $\Pic(\Sbb_{\K(1)}/p^k)=\Z/2(p-1)p^{k-1}$.
	\end{thm} 
	\begin{proof}
	As $\cd_{p}\Zpx=1$, the $E_2$-page of the descent spectral sequence in \Cref{thm:picss}
		\begin{equation*}
			E_2^{s,t}=\H_c^s(\Zpx;\pi_t\cPic(\KU/p^k))\Longrightarrow \pi_{t-s}(\Sbb_{\K(1)}/p^k), t-s\ge 0
		\end{equation*}
		has at most  two non-zero entries when $t-s\ge0$ and we have an extension on the $0$-th stem on the $E_\infty$-page:
		\begin{equation*}
			\begin{tikzcd}
				0\to \H^1_c(\Zpx;\pi_0(\KU/p^k)^\times)=\zx{p^k}\rar & \Pic_{\K(1)}(\Sbb_{\K(1)}/p^k)\rar & \Pic_{\K(1)}(\KU/p^k)^{\Zpx}=\Z/2\to0. 
			\end{tikzcd}
		\end{equation*}
		The homotopy groups of $\Sbb_{\K(1)}/p^k$ are $2(p-1)p^{k-1}$-periodic, so it follows that short exact sequence above does not split and we have $\Pic_{\K(1)}(\Sbb_{\K(1)}/p^k)= \Z/2(p-1)p^{k-1}$.
	\end{proof}
	At \(p=2\), we first compute $\Pic(\KO/2^k)$ by descending from $\Pic(\KU/2^k)=\Z/2$.
	\begin{thm}\label{thm:PicKO2k}
		When $k\ge 6$, $\Pic(\KO/2^k)=\Z/8$.
	\end{thm}
	\begin{proof}
		By \Cref{thm:picss} (\cite[(3-5)]{MS_Picard}), we have a descent spectral sequence:
		\begin{equation}\label{eqn:PicSS_KO2k}
			\!^{\Pic}E_2^{s,t}=\H^s(C_2;\pi_{t}\cPic(\KU/2^k))\Longrightarrow \pi_{t-s}\cPic(\KO/2^k).
		\end{equation}
		There is a similar homotopy fixed point spectral sequence:
		\begin{equation}\label{eqn:HFPSS_KO2k}
			\!^{\HFP}E_2^{s,t}=\H^s(C_2;\pi_{t}(\KU/2^k))\Longrightarrow \pi_{t-s}(\KO/2^k).
		\end{equation}
		Explicit computations in group cohomology of $C_2$ yield:
		\begin{equation*}
			\!^{\HFP}E_2^{s,t}=\begin{cases}
				\Z/2, & s=0,t\equiv 2\mod 4,\\& \text{or }s>0, t \text{ even};\\
				\Z/2^k, & s=0,4\mid t;\\
				0,&\text{else. }
			\end{cases}
			\qquad\qquad
			\!^{\Pic}E_2^{s,t}=\begin{cases}
				\Z/2, & t=0,\\
				&\text{or }s=0, t\equiv 3\mod 4,t>0,\\ 
				&\text{or }s\ge 1, t\ge 2 \text{ odd};\\
				(\Z/2^k)^\times, &s=0, t=1;\\
				\Z/2^k, & s=0, t\equiv 1\mod 4,t>1;\\
				\Z/2\oplus\Z/2, &s>0, t=1;\\
				0,&\text{else. }
			\end{cases}
		\end{equation*}
		Comparing with the homotopy fixed spectral sequence for $\KO\simeq \KU^{hC_2}$ and \cite[Comparison Tool 5.2.4]{MS_Picard}, we then obtain the $E_2$-pages and $d_3$-differentials in the HFPSS \eqref{eqn:HFPSS_KO2k} for $\KO/2^k$ and the DSS \eqref{eqn:PicSS_KO2k} for $\cPic \left(\mathrm{KO}/2^k\right)$ as below in \Cref{fig:sseq_KO2k}, except for the $d_3$-differential in red. The differentials and the extension problem on the $t-s=-1$ stem that are irrelevant to the computation of $\Pic(\KO/2^k)$ will be explained separately in \Cref{rem:PicSS_KO2k}. 
		
		\DisableQuotes
		\begin{figure}[ht]
			\begin{subfigure}[b]{0.48\linewidth}
				\centering
				\begin{sseqpage}[
					Adams grading,
					,x range={-2}{8},y range={0}{8},scale=0.57,classes={fill},tick style={font=\small,
						inner sep=0pt,
						outer sep=0pt,}]
					\fill[color=gray!30] (-0.5,-1) -- (-0.5,3.5) -- (-1.5,3.5) -- (-1.5,4.5) -- (-2.5,4.5) -- (-2.5,8.5) -- (8.5,8.5) -- (8.5,-1) --cycle;
					\foreach \x in {-18,-16,...,12} \foreach \y in {1,2,...,12} {\class(\x+\y,\y)}
					\foreach \x in {-12,-8,...,8}{
						\class[circlen=3](\x,0)
						\class (\x+2,0)
					}
					\foreach \x in {-12,-4} \foreach \y in {0,1,...,9} {\d3(\x+\y,\y) \d3(\x+\y+2,\y)}
					\foreach \y in {0,1,...,6} {\d3(4+\y,\y) \d3(6+\y,\y)}
				\end{sseqpage}
				\caption{\centering $E_2$-page and differentials of \\  the HFPSS \eqref{eqn:HFPSS_KO2k} for $\KO/2^{k}$.}
			\end{subfigure}
			\begin{subfigure}[b]{0.48\linewidth}
				\begin{sseqpage}[
					Adams grading,
					,x range={-1}{9},y range={0}{8},scale=0.57,classes={fill},tick style={font=\small,
						inner sep=0pt,
						outer sep=0pt,}]
					\fill[color=gray!30] (0.5,-1) -- (0.5,3.5) -- (-0.5,3.5) -- (-0.5,4.5) -- (-1.5,4.5) -- (-1.5,8.5) -- (9.5,8.5) -- (9.5,-1) --cycle;
					\node[inner sep=0pt, pin={[pin distance=1.5ex,font=\scriptsize]above:{[\beta]}}] at (-1,6) {};
					\foreach \x in {3,5,...,19} \foreach \y in {1,2,...,15} {\class({\x-\y},\y)}
					\foreach \y in {1,...,5} {\class({1-\y},\y)
						\class({1-\y},\y)
					}
					\node[inner sep=0pt, pin={[pin distance=1.5ex,font=\scriptsize]below left:{[\alpha]}}] at (0,1,1) {};
					\node[inner sep=0pt, pin={[pin distance=1.5ex,font=\scriptsize]above:{[-1]}}] at (0,1,2) {};
					\foreach \y in {0,1,...,12} {\class({-\y},\y)
					}
					\class[fill={none},draw={none},"\bigtimes"](1,0)
					\foreach \x in {3,7} {\class(\x,0)}
					\class[circlen=3](5,0)
					\class[circlen=3](9,0)
					\d[red,thick]3 (0,1,1)
					\foreach \x in {1,2,...,7}{\d3(\x,\x+1)}
					\foreach \x in {0,1,...,5}{\d3(\x+5,\x)}
					\foreach \x in {0,1,2,3}{\d3(\x+7,\x)}
					\foreach \x in {1,2,...,5}{\d3(\x,\x+3)}
					\d3(-1,8)
					\d2(-1,1)
					\d3(-1,2,2)
				\end{sseqpage}
				\caption{ \centering $E_2$-page and differentials of 
					\\ the DSS \eqref{eqn:PicSS_KO2k} for $\cPic(\KO/2^k)$, $k\ge 6$. \label{fig:PicSS_KO2k_E2}} 
			\end{subfigure}\\
			\begin{subfigure}[b]{0.48\linewidth}
				\begin{sseqpage}[
					Adams grading,
					,x range={-2}{8},y range={0}{8},classes={fill},scale=0.57,classes={fill},tick style={font=\small,
						inner sep=0pt,
						outer sep=0pt,},grid=crossword]
					\foreach \x in {0,8} {\class[circlen=3](\x,0)
						\class(\x+2,0)
						\class[circlen=2](\x+4,0)
						\class(\x+1,1)
						\class(\x+3,1)
						\class(\x+2,2)
						\class(\x+4,2)
						\structline[dashed] (\x+4,0) (\x+4,2)
					}
				\end{sseqpage}
				\caption{\centering $E_\infty$-page of the HFPSS \eqref{eqn:HFPSS_KO2k} for $\KO/2^{k}$,  $k\ge 2$.
					\protect\footnotemark}
			\end{subfigure} \begin{subfigure}[b]{0.48\linewidth}
				\begin{sseqpage}[
					Adams grading,
					,x range={-1}{9},y range={0}{8},classes={fill},scale=0.57,classes={fill},tick style={font=\small,
						inner sep=0pt,
						outer sep=0pt,},grid=crossword]
					\class[fill={none},draw={none},"\bigtimes"](1,0)
					\class(-1,2)
					\class[white](-1,2)
					\class[circlen=3](9,0)
					\class[circlen=2](5,0)
					\class(3,0)
					\class(2,1)
					\class(4,1)
					\class(3,2)
					\class(5,2)
					\class(0,0)
					\class[white](0,1)
					\class(0,1)
					\class(0,3)
					\class(-1,6)
					\structline[dashed] (5,0) (5,2)
					\structline[dashed] (0,0) (0,1,2)
					\structline[dashed] (0,1,2) (0,3)
					\structline[dashed] (-1,2,1) (-1,6)
				\end{sseqpage}
				\caption{\centering $E_\infty$-page of the DSS \eqref{eqn:PicSS_KO2k} for $\cPic(\KO/2^k)$, $k\ge 6$.\label{fig:PicSS_KO2k_Einfty}}
			\end{subfigure}
			\caption[The HFPSS and the Picard DSS for $\KO/2^{k}$]{\centering The HFPSS and the Picard DSS for $\KO/2^{k}$, \\Adams grading, $\bullet=\Z/2$, \protect	\tikz[baseline=-0.6ex]{
					\draw (0,0) circle (0.8ex); 
					\fill (0,0) circle (0.4ex);
				}~$=\Z/2^{k-1}$, \protect	\tikz[baseline=-0.6ex]{
					\draw (0,0) circle (1.2ex); 
					\draw (0,0) circle (0.8ex); 
					\fill (0,0) circle (0.4ex);
				}~$=\Z/2^k$, $\bigtimes=(\Z/2^k)^\times$.	\label{fig:sseq_KO2k}	} 
		\end{figure}		
		\footnotetext{When $k=1$, there are no extension problems on the ($8k+4$)-stems since then $E_\infty^{8k+4,0}=\Z/2^{1-1}=0$. Instead, there are non-trivial extensions on the ($8k+2$)-stems in that case.}This differential in red is supported at $\!^{\Pic} E_2^{1,1}=\H^1(C_2;(\Z/2^{k})^\times)\cong \hom(C_2,(\Z/2^k)^\times)\cong \Z/2\oplus \Z/2$.  Identifying this group cohomology with $2$-torsion elements in $(\Z/2^k)^\times$, then it is generated by $-1$ and $\alpha=1+2^{k-1}\in \ker [(\Z/2^k)^\times\to (\Z/2^{k-1})^\times]$.  Comparing with the DSS for $\Pic(\KO)$ in \cite{MS_Picard}, we know that $\{\pm 1\}\subseteq E_2^{1,1}$ and the $E_2^{3,3}$-terms are permanent cycles on the zero-stem, since they represent suspensions of $\KO/2^{k}$. In particular, the non-zero permanent cycle in $E_\infty^{3,3}$ is represented by $\Sigma^4\KO/2^k$.  For the sake of this theorem, it is enough to show $\alpha$ is \emph{not} a permanent cycle. 
		
		Suppose $\alpha$ is a permanent cycle, then it is represented by some $X\in \Pic(\KO/2^k)$.  Note $\KO/2^{k-1}$ has an $\Ebb_2$-ring structure when $k\ge 6$ and hence $\Pic(\KO/2^{k-1})$ can also be computed from a similar descent spectral sequence \eqref{eqn:PicSS_KO2k}. As $\alpha\in \ker [\H^1(C_2;(\Z/2^{k})^\times)\to \H^1(C_2;(\Z/2^{k-1})^\times)]$,  the base change $\KO/2^{k-1}\otimes_{\KO/2^k} X$ has descent filtration at least $3$ and is hence equivalent to either $\KO/2^{k-1}$ or $\Sigma^4\KO/2^{k-1}$.  In particular,  $\KO/2^{k-3}\otimes_{\KO/2^k} X$ is  either $\KO/2^{k-3}$ or $\Sigma^4\KO/2^{k-3}$.\footnote{Note that $\KO/2^{k-3}$ might not have an $\Ebb_2$-ring structure under the assumption $k\ge 6$. So we need to first consider base change to $\KO/2^{k-1}$. However, as we will see below, we have to consider the base change  $\KO/2^k\to \KO/2^{k-3}$ in order to apply  \Cref{cor:pic_inj}.} This puts us in the situation of \Cref{cor:pic_inj}. Notice that from \cite{Burkland_mult_Moore} we have
		\begin{itemize}
			\item $\KO/8= \KO\otimes \mathbb{S}/8$ admits an $\Ebb_1$-algebra structure.
			\item $\KO/2^k$ has compatible $\Ebb_2$-ring structures when $k\ge 5$.
		\end{itemize} 
		Consider the map of ring spectra  $\KO/2^{k}\to \KO/2^{k-3}$ when $k\ge 6$. It sits in a pullback square of rings:
		\begin{equation*}
			\begin{tikzcd}
				\KO/2^k\rar\dar\ar[dr,phantom,"\lrcorner", very near start]& \KO/2^3\dar\\
				\KO/2^{k-3}\rar& \KO/2^{k-3}\otimes_{\KO}\KO/2^3,
			\end{tikzcd}
		\end{equation*}
		where $\KO/2^{k-3}\otimes_{\KO}\KO/2^3\simeq \KO/8\oplus\Sigma \KO/8$ as a $\KO$-module spectrum.
		Consider the map of Milnor contexts $(\KO/2^{k-3},\KO/2^{k-3}\otimes_{\KO}\KO/2^3,\KO/2^{3}) \to (\KO/2^{k-3},\KO/2^3,\KO/2^{3})$. Applying \Cref{prop:pullbackupperbound}, with $N=R,N'=S$, we learn that the relative module space $\Perf(\KO/2^k\mid \KO/2^{k-3})$ is a collection of components inside the fiber of the map $$\Omega^{\infty}(\KO/2^{k-3}\otimes_{\KO}\KO/2^3) \to \Omega^{\infty}(\KO/2^3).$$ However, this fiber is $\Omega^{\infty}(\Sigma \KO/2^3)$, which is connected.
		
		Thus if $Y\in \Pic(\KO/2^{k})$ satisfies  $\KO/2^{k-3}\otimes_{\KO/2^{k}}Y\simeq\Sigma^{N}\KO/2^{k-3}$, then $Y\simeq \Sigma^N\KO/2^{k}$.  This implies that the hypothetical spectrum $X\in \Pic(\KO/2^k)$ represented by $\alpha$ is equivalent to $\KO/2^k$ or $\Sigma^4\KO/2^{k}$, since it is so after base change to $\KO/2^{k-3}$. However, this contradicts the assumption that $X$ has descent filtration $2$. Therefore, $\alpha$ \emph{cannot} be  a permanent cycle. 
		
		At this point, we can already conclude $\Pic(\KO/2^k)=\Z/8$, since only remaining three terms in the $t-s=0$ stem in the descent spectral sequence \eqref{eqn:PicSS_KO2k} in \Cref{fig:PicSS_KO2k_E2} are coming from that for $\Pic(\KO)=\Z/8$. 
	\end{proof}
	\EnableQuotes
	We include a proof of the differential supported by $\alpha\in \!^{\Pic} E_2^{1,1}$ for completeness. 
	\begin{prop}
		The class $[\alpha]\in \!^{\Pic} E_2^{1,1}$ supports a non-zero $d_3$-differential when $k\geq 6$. 
	\end{prop}
	\begin{proof}
		We assume that $k\geq 6$ so that $\cPic$ is an $\EE_2$-space and  the target of the differential in the spectral sequence makes sense as a group.
		From \Cref{fig:PicSS_KO2k_E2}, the possible non-zero differential supported by $\alpha$ is either  $d_3$ or $d_5$.  Suppose $d_3(\alpha)=0$, then $d_5(\alpha)$ is the non-zero element  $\beta\in \!^{\Pic} E_2^{6,5}=\H^6(C_2;\pi_5\cPic(\KU/2^k))\cong \H^6(C_2;\pi_4(\KU/2^k))$, where $C_2$ acts trivially on $\pi_5\cPic(\KU/2^k)\cong \pi_4(\KU/2^k)=\Z/2^k$.  
		
		Let $q_*$ be the induced map of $\cPic(\KO/2^k)\to \cPic(\KO/2^{k-1})$ on the descent spectral sequences \eqref{eqn:PicSS_KO2k}. Then group cohomology computations implies that $q_*(\beta)\ne 0\in \!^{\Pic} E_2^{6,5}(\KO/2^{k-1})=\H^6(C_2;\pi_5\cPic(\KU/2^{k-1}))$, whereas $q_*(\alpha)=0$ by the definition of $\alpha$.  From the naturality of differentials in the descent spectral sequence, this can only happen if $q_*(\beta)$ is hit by a shorter differential in the descent spectral sequence for $\cPic(\KO/2^{k-1})$. The only possible candidate to support this shorter differential, $^{\Pic} E_2^{3,3}(\KO/2^{k-1})$, is  known to be a permanent cycle however. Consequently, it is not possible to have $d_3(\alpha)=0$ and $d_5(\alpha)=\beta$. We have therefore proved the element $\alpha\in\!^{\Pic} E_2^{1,1}(\KO/2^{k})$ supports a $d_3$-differential in red in \Cref{fig:PicSS_KO2k_E2}.
	\end{proof}
	\begin{rem}
		Just like the class $\alpha$, the target of the $d_3$-differential $d_3(\alpha)\in\!^{\Pic} E_2^{4,3}=\H^4(C_2;\pi_2(\KU)/2^k)$ is killed upon base change to $\H^4(C_2;\pi_2(\KU)/2^{k-1})$. 
	\end{rem}
	\begin{rem}[$-1$-stem on the Picard DSS and the relative Brauer group]\label{rem:PicSS_KO2k}
		In \Cref{fig:PicSS_KO2k_E2}, the $d_2$-differential supported at $E_2^{1,0}$ and $d_3$-differential supported at $E_2^{2,1}$ can be deduced from the corresponding differentials on the DSS for $\cPic(\KO)$. This follows from the computation of the first two $C_2$-equivariant $k$-invariants of the Picard space $\cPic(\KU)$ in  \cite[Proposition 7.14]{Gepner-Lawson_Brauer}. In \Cref{fig:PicSS_KO2k_Einfty}, the extension problem on the $t-s=-1$ stem can be deduced from the corresponding extension problem on the DSS for $\cPic_{\K(1)}(\KO^\wedge_2)$, which was proved in \cite[Theorem 5.3]{Mor_relBr} as a computation of the relative Brauer group $\mathrm{Br}(\KO^\wedge_2\mid \KU^\wedge_2)$. See \cite[Figure 2]{Gepner-Lawson_Brauer}, \cite[Figure 1]{AMS_Brauer}, and  \cite[Figure 3]{Mor_relBr} for related computations of differentials and extension problems on the $-1$-stems on the DSS of $\cPic(\KO)$ and $\cPic_{\K(1)}\left(\KO^\wedge_2\right)$.
		
		In particular, our computation implies that $\mathrm{Br}(\KO/2^k\mid\KU/2^k)=\Z/4$ when $k\ge 6$.  Note that $k=6$ is the smallest power that $\Sbb/2^k$ is known to admit an $\Ebb_3$-ring structure by \cite{Burkland_mult_Moore}, which is needed to define the (relative) Brauer group. 
	\end{rem}
	\begin{thm}\label{thm:Pic_S02k}
		When $p=2$ and $k\ge 6$, we have
		\begin{equation*}
			\Pic(\Sbb_{\K(1)}/2^k)\cong \Z/2^{k-1}\oplus \Z/4\oplus \Z/2.
		\end{equation*}
	\end{thm}
	\begin{proof}
		The profinite descent spectral sequence in \Cref{thm:picss} for the Picard groups of the $(1+4\Z_2)$-Galois extension $\mathbb{S}_{\K(1)}/2^k\to \KO/2^k$ has $E_2$-page:
		\begin{equation*}
			E_2^{s,t}=\H_c^s(1+4\Z_2;\pi_t\cPic_{\K(1)}(\KO/2^k))\Longrightarrow \pi_{t-s}\cPic(\Sbb_{\K(1)}/2^k), \quad t-s\ge 0.
		\end{equation*}
		As $\cd_2(1+4\Z_2)=1$, the $E_2$-page is concentrated in the $s=0,1$-lines and hence the spectral sequence collapses on this page. This yields an extension problem on $0$-stem of the $E_\infty$-page, which is compatible with the one on the $E_\infty$-page of the descent spectral sequence for $\Pic(\Sp_{\K(1)})$ in \cite{HMS_picard}.
		\begin{equation*}
			\begin{tikzcd}[row sep=small]
				0\rar &[-12 pt]\H_c^1(1+4\Z_2;\pi_1\cPic_{\K(1)}(\KO^\wedge_2)) \rar\dar[symbol=\cong ]& \Pic(\Sp_{\K(1)}) \rar \dar[symbol=\cong ]&\H_c^0(1+4\Z_2;\Pic_{\K(1)}(\KO^\wedge_2))\dar[symbol=\cong ]\rar&[-12 pt] 0\\
				0\rar&	\Z_2^\times\dar[->>,"f_1"]\rar["i"]&  \Z_2\oplus \Z/4\oplus \Z/2\rar\dar["f"]&\Z/8\dar[equal]\rar&0\\ [10pt]
				0\rar&	(\Z/2^k)^\times\rar\dar[symbol=\cong ] &{\color{red}??}\rar \dar[symbol=\cong ]&\Z/8\rar \dar[symbol=\cong ]&0\\ 
				0\rar &\H_c^1(1+4\Z_2;\pi_1\cPic_{\K(1)}(\KO/2^k)) \rar & \Pic(\Sbb_{\K(1)}/2^k) \rar &\H_c^0(1+4\Z_2;\Pic_{\K(1)}(\KO/2^k))\rar& 0
			\end{tikzcd}
		\end{equation*}
		By the snake lemma, the middle vertical map $f$ is surjective with kernel equal to  $i(\ker f_1)=i(1+2^k\Z_2)$.  The map $i$ sends a generator $g\in 1+4\Z_2\le \Z_2^\times$ to $(2,1,0)\in \Z_2\oplus \Z/4\oplus \Z/2\cong \Pic(\Sp_{\K(1)})$. Therefore $\ker f=i(\ker f_1)=2^{k-1}\Z_2$ when $k\ge 4$. From this, we conclude
		\begin{equation*}
			\Pic(\Sbb_{\K(1)}/2^k)\cong (\Z_2\oplus \Z/4\oplus \Z/2)/(2^{k-1}\Z_2)\cong \Z/2^{k-1}\oplus\Z/4\oplus\Z/2. \qedhere 
		\end{equation*}
	\end{proof}
	\begin{rem}
		When $k\ge 4$, Davis--Mahowald showed in  \cite[Proposition 2.3]{Davis-Mahowald1990}  that  the Moore spectrum $\mathbb{S}/2^k$ admits a self map $\Sigma^{2^{k-1}}\mathbb{S}/2^k\to \mathbb{S}/2^k$, which induces the multiplication-by-$u^{2^{k-2}}$ map on its $\KU$-cohomology groups. This is reflected in the $\Z/2^{k-1}$-summand in $\Pic(\Sbb_{\K(1)}/2^k)$.  
	\end{rem}
	\begin{rem}\label{rem:PicKO32}
		By \cite{Burkland_mult_Moore}, $\KO/32=\KO\otimes \Sbb/32$ and $\Sbb_{\K(1)}/32$ admit an $\Ebb_2$-ring structure, so their Picard groups are defined. In \Cref{thm:PicKO2k}, we are not able to determine $\Pic(\KO/32)$ -- it is either $\Z/8$ or $\Z/8\oplus\Z/2$ depending on whether the $d_3$-differential in red in \Cref{fig:PicSS_KO2k_E2} exists or not. Following \Cref{thm:Pic_S02k}, we have $\Pic(\Sbb_{\K(1)}/32)$ is either $\Z/16\oplus\Z/4\oplus\Z/2$ or its extension by $\Z/2$. 
	\end{rem}
	\emergencystretch=1em
	\printbibliography
\end{document}